\documentclass{amsart}
\usepackage[utf8]{inputenc}
\usepackage{amsmath,amssymb,amsthm, graphics, comment,bm}
\usepackage{mathtools}
\usepackage{amsthm}
\usepackage{diagmac2}
\usepackage{textcomp}
\usepackage{color}
\usepackage[T1]{fontenc}
\usepackage[alphabetic]{amsrefs}
\usepackage[all,cmtip,color,matrix,arrow]{xy}
\usepackage{yfonts}
\usepackage{tikz}
\usepackage{tikz-cd}
\usepackage{stmaryrd}
\usepackage{mathrsfs}
\usepackage{enumerate}
\usepackage{hyperref}
\hypersetup{
	colorlinks=true,
	linkcolor=blue,
	filecolor=magenta,      
	urlcolor=cyan,
}
\usepackage[mathscr]{euscript}
\usepackage[all,cmtip]{xy}
\usepackage{cleveref}

\calclayout

\newcommand{\bA}{\mathbb{A}}

\newcommand{\bE}{\mathbb{E}}

\newcommand{\bG}{\mathbb{G}}
\newcommand{\bL}{\mathbb{L}}
\newcommand{\bN}{\mathbb{N}}
\newcommand{\bP}{\mathbb{P}}

\newcommand{\bZ}{\mathbb{Z}}

\newcommand{\oH}{\operatorname{H}}

\newcommand{\cB}{\mathcal{B}}
\newcommand{\cC}{\mathcal{C}}
\newcommand{\cD}{\mathcal{D}}
\newcommand{\cE}{\mathcal{E}}

\newcommand{\cG}{\mathcal{G}}
\newcommand{\cH}{\mathcal{H}}
\newcommand{\cI}{\mathcal{I}}
\newcommand{\cJ}{\mathcal{J}}

\newcommand{\cL}{\mathcal{L}}

\newcommand{\cN}{\mathcal{N}}
\newcommand{\cO}{\mathcal{O}}
\newcommand{\cP}{\mathcal{P}}

\newcommand{\cS}{\mathcal{S}}

\newcommand{\cU}{\mathcal{U}}
\newcommand{\cV}{\mathcal{V}}
\newcommand{\cW}{\mathcal{W}}
\newcommand{\cX}{\mathcal{X}}
\newcommand{\cY}{\mathcal{Y}}
\newcommand{\cZ}{\mathcal{Z}}

\newcommand{\stProj}{\operatorname{Proj}}

\newcommand{\E}{\operatorname{Ext}}
\newcommand{\spec}{\operatorname{Spec}}

\newcommand{\Hom}{\operatorname{Hom}}

\newcommand{\Gm}{\bG_{{\rm m}}}

\newcommand{\Id}{\operatorname{Id}}

\newcommand{\Isom}{\operatorname{Isom}}

\newcommand{\bmu}{\bm{\mu}}
\newcommand{\Out}{\operatorname{Out}}
\newcommand{\pr}{{\rm pr}}
\newtheorem{theorem}{Theorem}[section]
\newtheorem{Teo}[theorem]{Theorem}
\newtheorem*{Teo*}{Theorem}
\newtheorem{Lemma}[theorem]{Lemma}

\newtheorem{Cor}[theorem]{Corollary}
\newtheorem{example}[theorem]{Example}

\newtheorem{Prop}[theorem]{Proposition}
\newtheorem*{Ques*}{Question}

\theoremstyle{definition}
\newtheorem{Oss}[theorem]{Remark}
\newtheorem*{Oss'}{Remark}

\newtheorem{Def}[theorem]{Definition}
\newtheorem*{Def*}{Definition}

\newtheorem{Remark}[theorem]{Remark}

\newcommand{\marg}[1]{\normalsize{{\color{red}\footnote{{\color{blue}#1}}}{\marginpar[{\color{red}\hfill\tiny\thefootnote$\rightarrow$}]{{\color{red}$\leftarrow$\tiny\thefootnote}}}}}
\newcommand{\Gio}[1]{\marg{(Giovanni) #1}}

\begin{document}
\title{A criterion for smooth weighted blow-downs}
\author[V.~Arena]{V.~Arena}
\address[Veronica Arena]{DPMMS, University of Cambridge, UK}
	\email{va365@cam.ac.uk}
\author[A.~Di Lorenzo]{A.~Di Lorenzo}
	\address[Andrea Di Lorenzo]{Dipartimento di Matematica, Università di Pisa, Italy}
	\email{andrea.dilorenzo@unipi.it}
\author[G.~Inchiostro]{G.~Inchiostro}
	\address[Giovanni Inchiostro]{University of Washington, Seattle, Washington, USA}
	\email{ginchios@uw.edu}
\author[S.~Mathur]{S.~Mathur}
\address[Siddharth Mathur]{ University of Georgia, Athens, USA}
\email{siddharth.mathur@uga.edu}
\author[S.~Obinna]{S.~Obinna}
\address[Stephen Obinna]{University of Waterloo, Canada}
	\email{sobinna@uwaterloo.ca}
\author[M.~Pernice]{M.~Pernice}
\address[Michele  Pernice]{University of Washington, Seattle, Washington, USA}
	\email{mpernice@uw.edu}
	\maketitle
	\begin{abstract}
		We establish a criterion for determining when a smooth Deligne-Mumford stack is a weighted blow-up. More precisely, given a smooth Deligne-Mumford stack $\cX$ and a Cartier divisor $\cE \subset \cX$ such that (1) $\cE$ is a weighted projective bundle over a smooth Deligne-Mumford stack $\cY$ and (2) for every $y\in\cY$ we have $\cO_{\cX}(\cE)|_{\cE_y}\simeq \cO_{\cE_y}(-1)$, then there exists a contraction $\cX\to\cZ$ to a smooth Deligne-Mumford stack $\cZ$. Moreover, the stack $\cX$ can be recovered as a weighted blow-up along $\cY\subset \cZ$ with exceptional divisor $\cE$, and $\cZ$ is a pushout in the category of algebraic stacks. As an application, we show that the moduli stack $\overline{\mathscr{M}}_{1,n}$ of stable $n$-pointed genus one curves is a weighted blow-up of the stack of pseudo-stable curves. A key step is a reconstruction result for smooth Deligne-Mumford stacks that may be of independent interest. 
	\end{abstract}

 \section{Introduction}
 Blow-ups are arguably the most important kind of birational transformations between algebraic varieties. Indeed, any rational map between smooth varieties $X\dashrightarrow Y$ can be factored as a composition of blow-ups and blow-downs along smooth centers \cite{abramovich_factor}. Similarly, classical desingularizations can also be factored into a sequence of blow-ups.
 As such, it is natural to ask: given a variety $X$, when is it the blow-up of another variety $Y$? This is useful as usually $Y$ is easier to understand than $X$. For smooth surfaces, the answer is classical: Castelnuovo's theorem says that $X$ is the blow-up of a smooth surface $Y$ if an only if it has a rational $-1$-curve. 
 For higher dimensional varieties, there are similar results due to Kodaira \cite{kodaira1954kahler} and Grauert \cite{grauert1962modifikationen} (for contracting $n$-dimensional projective spaces), to Moishezon \cite{moishezon1966n} (in the complex manifold case) , and to Artin \cite{artin1970algebraization} (in even greater generality).

 In the realm of algebraic stacks, there is a natural generalization of blow-ups: \textit{weighted blow-ups} \cites{abramovich2019functorial, QR}. Rather than being defined by a single ideal sheaf, these are characterized by a sequence of ideal sheaves which satisfy a compatibility condition (see \Cref{def_weighted_embedding}).
 Note that two birational algebraic stacks need not be related by a sequence of blow-ups (e.g. root stacks), so the extra flexibility of weighted blow-ups yields better algorithms for resolution of singularities \cite{abramovich2019functorial}.
 
 Just as with classical blow-ups, a weighted blow-up $\cX\to \cZ$ is an isomorphism on the complement of a closed substack, which we call the \textit{reduced center} of the weighted blow-up.
 Thus, weighted blow-ups yield birational transformations between algebraic stacks, and sometimes they even introduce new stacky structure.
 Indeed, the exceptional divisor of a weighted blow-up is a weighted projective bundle over $\cY$: this is a stack which, locally over $\cY$, looks like the base times $\cP(a_1,\ldots,a_n)$. This is the quotient stack $[\bA^{n}\smallsetminus\{0\}/\bG_m]$ with $\bG_m$ acting linearly with weights $a_1,\ldots, a_n$. 
 Then a natural question is: \textcolor{black}{given an algebraic stack $\cX$, when is it the \emph{weighted} blow-up of an algebraic stack $\cZ$?}

 \begin{Teo}\label{Thm_intro}Let $\cX$ and $\cY$ be smooth \textcolor{black}{tame} and separated Deligne-Mumford stacks, and let $\pi:\cE\to \cY$ be a weighted projective bundle with positive dimensional fibers. Assume that there is a closed embedding $\cE\hookrightarrow \cX$ with
  $\cE$ a Cartier divisor in $\cX$, such that $\cO_\cX(\cE)|_\cE\cong\cO_{\cE}(-1)\otimes \pi^*\cL$ for a line bundle $\cL$ over $\cY$.
  
  Then there is a smooth, separated \textcolor{black}{and tame} Deligne-Mumford stack $\cZ$, with two maps $i: \cY\to \cZ$ and $p: \cX\to \cZ$ such that $i$ is a closed embedding and $p$ is a weighted blow-up with reduced center $\cY$. \textcolor{black}{Moreover, the resulting square is a pushout in algebraic stacks.}
 \end{Teo}

As a sample application, we show the following.
\begin{Teo}\label{thm:application intro}
    Let $\overline{\mathscr{M}}_{1,n}$ be the moduli stack of stable $n$-pointed curves of genus one, with $n\geq 2$.
    \begin{enumerate}
        \item there is a contraction $\overline{\mathscr{M}}_{1,n}\to \overline{\mathscr{M}}^*_{1,n}$ which sends the divisor $\Delta_{0,n}$ of curves with elliptic tails to a closed substack isomorphic to $\overline{\mathscr{M}}_{0,n+1}$. This contraction is a weighted blow-up.
        \item The stack $\overline{\mathscr{M}}_{1,n}^*$ is isomorphic to the moduli stack of pseudo-stable curves $\overline{\mathscr{M}}_{1,n}^{\rm{ps}}$ \cite[Sec. 5]{Sch}, hence $\overline{\mathscr{M}}_{1,n}$ is a weighted blow up of $\overline{\mathscr{M}}_{1,n}^{\rm{ps}}$.
    \end{enumerate}
\end{Teo}
See also \cite{Inc} for the case $n=2$, and \cite{BDL} for further applications of this criterion.

\textcolor{black}{Finally, in \Cref{section_appendix} Stephen Obinna studies how \Cref{theo:main} changes if one replaces $\cO_\cE(-1)$ with $\cO_\cE(-m)$ (see \Cref{thm_appendix}).}
\subsection{Strategy of proof}
Our strategy for proving the main theorem is the following: we first prove that the hypotheses of the main theorem imply the existence of a contraction at the level of coarse spaces. 

We leverage this first result to prove the main theorem in the simpler case of $\cY=\spec(A)$ an affine scheme and $\cE\simeq \cP_A(a_1,\ldots,a_n)$: in fact $\cY$ being a scheme implies that locally around $\cE$ the stack $\cX$ is a scheme away from $\cE$, hence the contracted \emph{stack} can be obtained by gluing $\cX\smallsetminus\cE$ and a neighbourhood of $\spec(A)$ in the contracted coarse space. More effort is needed in order to prove the smoothness of this stack, and that $\cX$ is a weighted blow-up.

We then extend the previous result to an equivariant setting, i.e. we suppose that a finite group $G$ acts on all the stacks involved, and all the maps are equivariant, and we prove that the $G$-action extends to the contracted stack. This easily implies our main theorem in the case $\cY=[\spec(A)/G]$ and $\cE=[\cP_A(a_1,\ldots,a_n)/G]$. 

We lift then the contraction of the coarse spaces to a contraction of smooth Deligne-Mumford stacks, thus proving our main theorem. We start by showing that the contracted coarse space has only finite quotient singularities, hence it can be regarded as the coarse space of a smooth Deligne-Mumford stack $\cZ$. This requires extending stacky structure along open immersions and we prove results along these lines.
%\begin{Teo}[\Cref{thm_DM_stack_is_determined_by_cms_and_codim1}]\label{thm_reconstruction_intro}
   %  Let $X$ be a separated and connected scheme of finite type over $\spec(k)$, with finite tame quotient singularities, let $U$ be an open subscheme of $X$ with complement of codimension $\geq 2$. Assume that $\cU$ is a smooth Deligne-Mumford gerbe with coarse moduli space $U$. Then there is a unique smooth Deligne-Mumford gerbe $\cX$ such that its coarse moduli space is isomorphic to $X$, and with an open embedding $\cU\to \cX$ such that the compositions $\cU\to \cX\to X$ and $\cU\to U \to X$ agree.
% \end{Teo} To construct the contraction, we work then \'{e}tale locally on the Deligne-Mumford stack $\cZ$, and then we leverage the fact that smooth and separated Deligne-Mumford stacks locally always look like $[W/H]$ for a finite group $H$ acting on an affine scheme $W$, in order to apply the results of the previous steps. The locally-defined contractions glue by the uniqueness of the local constructions.

\subsection{Organization of the manuscript}
The manuscript is organized as follows.  \Cref{section_introduction_to_wblowups} begins with a brief introduction to weighted blow-ups and their main properties.
In \Cref{section_prelims} we prove an extension result for Deligne-Mumford stacks (see \Cref{thm_DM_stack_is_determined_by_cms_and_codim1}).
In \Cref{section_contraction_on_cms} we construct an algebraic space $Z$, which will be the coarse moduli space of $\cZ$, using \cite{artin1970algebraization}. We also prove that $Z$ is normal, which will be useful later, as it will allow us to apply Zariski's main theorem to prove that certain maps are isomorphisms. 
\Cref{section_local_analisis} contains the most technical results: roughly speaking, we prove the case where the coarse moduli space of $\cY$ is affine and finite over an affine scheme $B$. The core argument is \Cref{thm_affine_contraction}, which is an adaptation of the proof of Castelnuovo's theorem in \cite{kol_resol}.
In \Cref{section_pf_thm} we prove the main result. Using the results of \Cref{section_local_analisis} we construct $\cZ$ \'etale locally over $Z$. By using \Cref{thm_DM_stack_is_determined_by_cms_and_codim1} and Zariski's main theorem we prove that the local constructions glue, and give the desired algebraic stack $\cZ$. 
In \Cref{section_application_to_the_case_of_curves} we give an application of our main result to some moduli stacks of curves: namely, let $\overline{\mathscr{M}}_{1,n}$ be the moduli stack of stable $n$-pointed curves of genus one. Then we use our theorem to construct a contraction $\overline{\mathscr{M}}_{1,n}\to \overline{\mathscr{M}}^*_{1,n}$ which sends the divisor $\Delta_{0,n}$ of curves with elliptic tails to a closed substack isomorphic to $\overline{\mathscr{M}}_{0,n+1}$. The stack $\overline{\mathscr{M}}_{1,n}^*$ turns out to be isomorphic to the moduli stack of pseudo-stable curves $\overline{\mathscr{M}}_{1,n}^{\rm{ps}}$, thus exhibiting $\overline{\mathscr{M}}_{1,n}$ as  a weighted blow up of $\overline{\mathscr{M}}_{1,n}^{\rm{ps}}$.
Finally in \Cref{section_appendix} Stephen Obinna shows how \Cref{Thm_intro} changes if we assume instead that $\cO_\cE(-1)$ is replaced with $\cO_\cE(-m)$.
 
 \subsection*{Conventions} We work over a field $k$ which is algebraically closed. \textcolor{black}{Unless otherwise stated, all our algebraic stacks will be with affine stabilizers and locally excellent.}\\
 
\subsection*{Acknowledgements} We thank Dan Abramovich, Jarod Alper, Elden Elmanto, Tyson Klingner, Martin Olsson, Mauro Porta and David Rydh for helpful conversations. We also thank the anonymous referee for their useful suggestions. Michele Pernice was supported by the Knut and Alice Wallenberg foundation 2021.0291. Veronica Arena and Stephen Obinna were supported in part by funds from BSF grant 2018193 and NSF grant DMS-2100548. Siddharth Mathur was supported
by FONDECYT Regular grant No. 1230402.

  \section{Introduction to weighted blow-ups}\label{section_introduction_to_wblowups}

This section is organized in two subsections, where we first introduce weighted blow-ups in general, and then we specialize to \textit{regular} weighted blow-ups: a particular type of weighted blow-ups.
\subsection{Weighted embeddings and weighted blow-ups}
This introduction to weighted blow-ups is based on the work of Quek and Rydh \cite{QR}.
Intuitively, an ordinary blow-up is constructed by replacing the center with a projective bundle which parameterizes the normal directions to the center in the ambient space. When considering weighted blow-ups, the same intuition holds with the added subtlety that each normal direction has a (positive integer) weight. 
This intuition highlights that the datum of a closed embedding, which is all we need in order to define an ordinary blow-up, is no longer enough. In fact, for each choice of basis for the normal directions, we need to specify the positive integer weight. 
This data is given by a \textit{weighted embedding}. 

\begin{Def}[\cite{QR}*{Def. 3.1.1}]\label{def_weighted_embedding}
    A weighted embedding $Y_{\bullet} \hookrightarrow X$ is defined by a sequence of closed embeddings $\{ Y_n=V(I_n) \hookrightarrow X \}_{n \ge 0}$ such that: \begin{itemize}
        \item $I_0 \supset I_1 \supset \dots \supset I_n \supset \dots$
        \item $I_n I_m \subset I_{n+m}$
        \item Locally in the smooth topology on $X$, there exists a sufficiently large positive integer $d$ such that for all integers $n \ge 1$, $$I_n=\left(I_1^{l_1}I_2^{l_2} \cdots I_d^{l_d} \; : \; l_i \in \bN, \ \sum_{i=1}^d il_i=n\right)$$ in which case, we say $I_\bullet$ is generated in degrees $\le d$. 
    \end{itemize} 
\end{Def} 

The three conditions above are the conditions that the filtration of ideal sheaves needs to satisfy in order to be a Rees algebra. 

\begin{Def}The weight of a local coordinate will be the maximum index $n$ such that $I_n$ contains said local coordinate. \end{Def}

\begin{example}\label{embeddingoftheorigin}
    Consider the sequence of ideals 
    $$k[x,y] \supset (x,y) \supset (x^2, y) \supset (x^3, xy, y^2) \supset (x^4, y^2, x^2y) \supset \dots$$

    These ideals define a weighted embedding $0_\bullet$ of the origin in $\bA^2$.
    In this case, the weights will be $1$ for $x$ and $2$ for $y$, and this weighted embedding will define the weighted blow-up of the origin in $\bA^2$ with weights $(1,2)$. 

\end{example}

\begin{Def}[Weighted blow-up]
Let $Y_\bullet \hookrightarrow X$ be a weighted embedding defined by the ideal sheaves $I_n$ and let $I_\bullet=\bigoplus I_nt^n \subset \cO_X[t]$. Let $I_{+}$ be the ideal generated by $t$.
The weighted blow-up of $X$ at $Y=V(I_1)$ is defined as the morphism $$Bl_{Y}X:=\stProj_X\left(I_\bullet\right)=[(\spec_X(I_\bullet) \setminus V(I_+))/\bG_m] \rightarrow X$$ where $\stProj$ is the stack-theoretic proj. 
\end{Def}

The same morphism is obtained by considering the extended Rees algebras:
\begin{Remark}\label{def_wblowup_algebra_in_all_degrees}
Let $Y_\bullet \hookrightarrow X$ be a weighted embedding defined by the ideal sheaves $I_n$ and let $I_\bullet^{ext}=\bigoplus_{n \in \bZ} I_nt^n \subset \cO_X[t^{\pm 1}]$ with $I_n = \cO_X $ for $n \le 0$. The weighted blow-up of $X$ at $Y$ can also be defined as the morphism $\stProj_X\left(I_\bullet^{ext}\right)\rightarrow X$, see \cite{QR}*{Rmk. 3.2.4}.
\end{Remark}

The definition of weighted blow-up via $I_\bullet^{ext}$ as in \Cref{def_wblowup_algebra_in_all_degrees} will often make for an easier description, as in the following example. 

\begin{example}\label{blowupoftheorigin} Given the weighted embedding in \Cref{embeddingoftheorigin}, we can explicitly write $Bl_{0}\bA^2 \xrightarrow{\pi}\bA^2$ via  $$Bl_{0}\bA^2 \cong \stProj \left(\frac{k[x,y][s,x',y']}{(x-sx', y-s^2y')}\right)$$
where $s$ has degree $-1$, $x'$ has degree $1$ and $y'$ has degree $2$.
Similarly, the weighted blow-up of the origin in $\bA^n$ with weights $(a_1,\dots a_n)$ will be $$Bl_{0}\bA^n \cong \stProj\left(\frac{k[x_1,\dots x_n][s,x_1',\dots x_n']}{(x_1-s^{a_1}x_1',\dots, x_n-s^{a_n}x_n')}\right)$$ with $x_i'$ in degree $a_i$ and $s$ in degree $-1$.
\end{example}
\begin{Remark}
    \Cref{blowupoftheorigin} can also be obtained as in \cite{Inc}*{Sec. 2.2}. In particular, it is the open substack of $[\bA^{n+1}/\Gm]$, given by the complement of the points of the form $(0,...,0,s)$, where the action of $\Gm$ on the coordinates of $\bA^{n+1}$ is with weights $(a_1,...,a_n,-1)$.
\end{Remark}
Lastly, we cite the two main properties of weighted blow-up, \cite[Thm. 3.2.9 and Cor. 3.2.14]{QR}. 
\begin{theorem}[Universal property of weighted blow-ups]\label{thm_univ_property_wblowup}
    Let $Y_\bullet$ be weighted embedding induced by the Rees algenbra $I_\bullet$, and let $\pi: X'=Bl_YX \rightarrow X$ the induced weighted blow-up. \begin{enumerate}
        \item For every $n \in \bN$ we have $\pi^{-1}(I_n)\cdot \cO_{X'}\subset I^n_E$ which is (locally on X) an equality for a sufficiently divisible $n$;
        \item Let $f: T \rightarrow X$ be a morphism such that $U:=T \setminus f^{-1}(Y)$ is schematically dense. The groupoid of factorizations through $\pi$ is equivalent to the set of effective Cartier divisors $D$ on $T$ such that $f^{-1}(I_n) \cdot \cO_T \subset I^n_D$ for all $n$, with equality for sufficiently divisible $n$ (locally on T). If $f=\pi \circ g$, then $D=g^{-1}(E).$
    \end{enumerate}  
\end{theorem}
\begin{theorem}[Flat base change]\label{flat_bc_wblowups}
    Let $f: X' \rightarrow X$ be a flat morphism. Let $Y_\bullet \hookrightarrow X$ be a weighted embedding with ideals $I_n$ and let $Y'_\bullet \hookrightarrow X'$ the induced weighted embedding with ideals $f^{-1}(I_n)\cdot \cO_{X'}$. Then the following square is Cartesian 
    $$\begin{tikzcd}
    Bl_{Y'}X' \arrow[d] \arrow[r] & Bl_{Y}X \arrow[d] \\
    X' \arrow[r, dashed, hook]   & X.                \end{tikzcd}$$
\end{theorem}
\subsection{Exceptional divisor of a weighted blow-up}
In this subsection we introduce a special kind of blow-up that will be of interest in this article.
\begin{Def}[Exceptional divisor]
The inclusion $I_{\bullet + 1}\subset I_{\bullet}$ induces an inclusion $\cO_{Bl_YX}(1) \to \cO_{Bl_Y X}$. We define $\cE$ as the Cartier divisor associated with this inclusion. In particular by construction we have $\cO_{Bl_YX}(-\cE) \cong \cO_{Bl_YX}(1)$.
\end{Def}

Note that, as we can see in the next example, it is no longer the case that the diagram
$$\begin{tikzcd}
\cE \arrow[r] \arrow[d] & Bl_{Y}X \arrow[d] \\
Y \arrow[r] & X
\end{tikzcd}$$
is  Cartesian unless all the weights are $1$ (and we are in the case of an ordinary blow-up).

\begin{example} Consider the blow-up of the origin in $\bA^2$ computed in \ref{blowupoftheorigin}. Then the exceptional divisor corresponds to the the locus $V(s)$. In particular $$\cE \cong \stProj\left(k[x',y']\right)$$ with $x'$ in degree $1$ and $y'$ in degree $2$, which is exactly the weighted projective space $\cP(1,2)$. 
   On the other hand, with the same degrees, $$Y \times_X Bl_YX \cong \stProj\left(\frac{k[x',y',s]}{(sx', s^2y')}\right)$$ which has some nonreduced structure.
\end{example}

Nonetheless, it is the case that the exceptional divisor and the fiber product $Y \times_X Bl_YX$ have the same reduced structure. 
The following observation will be useful later. 
\begin{Oss}\label{rmk_wblowup_of_0_in_An} In the example \Cref{blowupoftheorigin},
 the exceptional divisor can be obtained as the zero locus $V(s)$. It follows that $\pi_*\cO_{Bl_0\bA^n}(-k\cE)=I_n$ and hence
   $Bl_0\bA^n \cong \stProj(\bigoplus_k \pi_*\cO_{Bl_0\bA^n}(-k\cE))$.
\end{Oss}

%Going back to our intuition, we want to identify the exceptional divisor with something that parameterizes "weighted normal directions". In particular, we want to describe $\cE$ as the projectivization of some generalization of the normal bundle.  \begin{Def}\Gio{I would get rid of this definition. @ Veronica if you agree just get rid of it} We define the weighted conormal algebra as $\mathscr{C}_{Y_\bullet}X:=\bigoplus_{n \ge 0} I_n/I_{n+1}$ and the weighted normal cone $C_{Y_\bullet}X=\spec_X(\mathscr{C}_{Y_\bullet}X)$. \end{Def} It follows that the exceptional divisor is $\cE \cong \stProj_X(\bigoplus_{n \ge 0} I_n/I_{n+1})$.

In this manuscript we will only deal with a special type of weighted blow-up, those which in \cite{QR}*{Sec. 5.2} come from a \textit{regular} weighted embedding, and which we call a \textit{regular} weighted blow-up:
\begin{Def}
    The weighted blow-up $B\to X$ of $X$ at $Y$ is a \textit{regular} weighted blow-up if:
    \begin{enumerate}
        \item there are positive integers $a_1$, ..., $a_n$, and
        \item there is a smooth cover $X'\to X$ with a map $X'\to \bA^n$ which is flat at the preimage of the origin, 
    \end{enumerate}
    such that $B\times_X X'\cong X'\times_{\bA^n}Bl_0\bA^n$, where $Bl_0\bA^n$ is the weighted blow-up of 0 in $\bA^n$ with weights $a_1,...,a_n$ as in \Cref{blowupoftheorigin}.
\end{Def}
In other terms, a {regular} weighted blow-up is a weighted blow-up which, locally in the smooth topology on the center, is the flat pull-back of the weighted blow-up of the origin in $\bA^n$ of \Cref{blowupoftheorigin}.
Observe also that, from \cite[\href{https://stacks.math.columbia.edu/tag/062H}{Tag 062H}]{stacks-project}, the preimage of $x_1,...,x_n$ is a regular sequence in $X'$. So the center $Y$ has as ideal sheaf a regular sequence.
\begin{Remark}
    Let $X$ be a scheme and let $f_1,...,f_n$ a regular sequence in $\cO_X$. Then by \cite{Eisenbud}*{Ex. 18.18} the map $X\to \bA^n$ given by the $f_i$ is flat along the preimage of the origin in $\bA^n$.
    In particular, from \Cref{flat_bc_wblowups}, the weighted blow-up of $V(f_1,...,f_n)$ where $f_i$ has weight $a_i$ is a {regular weighted} blow-up, and it is the pull-back of the
    weighted blow-up of the origin in $\bA^n$, with weights $a_1,...,a_n$. So a weighted regular blow-up is the blow-up given by a sequence of ideal sheaves which in \cite{QR}*{Sec. 5.2} is called a regular weighted embedding.
\end{Remark}
\begin{Def}
    A \textit{weighted affine bundle} is a $\bG_m$-equivariant $\bA^n$-bundle $\cN \to Y$ where locally in the smooth topology $\bG_m$ acts linearly on $\mathbb{A}^n$ with positive weights $a_1,...,a_n \in \mathbb{Z}$. Moreover, a \textit{weighted projective bundle} over $Y$ is the stack quotient $\cP(\cN)=[\cN \setminus \cN_0/\bG_m]$, where $\cN_0$ is the zero section in $\cN$.
\end{Def}

\begin{example}\label{twistedweightedvectorbundle}
Let $\cN=\spec_Y(R)$ be a weighted affine bundle of rank 2 with weights (1,2), and suppose
there is a Zariski covering $Y=U\cup U'$ such that $R|_U\cong \cO_Y[x,y]$ and $R|_{U'}\cong \cO_Y[x',y']$ with $\Gm$ acting by multiplication with weight 1 on $x,x'$ and weight 2 on $y,y'$.
The transition functions for $\cN$ will map $$x' \mapsto \alpha x;  \quad y' \mapsto \beta y+\gamma x^2$$ with $\alpha, \beta \in \Gamma(U \cap U', \cO^\times_{U \cap U'})$, $\gamma \in \Gamma(U \cap U', \cO_{U \cap U'})$. 
\end{example}

\begin{Remark}
    From \cite{QR}*{Prop. 5.1.4}, the exceptional divisor of a weighted blow-up is a weighted projective bundle if and only if it comes from a {regular} weighted blow-up. In particular for us, as we only deal with {regular} weighted blow-ups, all the exceptional divisors will be weighted projective bundles over the center of the weighted blow-up.
\end{Remark}
We end this section with the following proposition:
\begin{Prop}\label{prop_checking_wblowup_smooth_locally}
    Let $\pi:\cX\to \cY$ be a morphism of algebraic stacks, and assume that:
    \begin{enumerate}
    \item There is a smooth cover $Y\to \cY$ such that $\pi_Y:\cX\times_{\cY}Y\to Y$ is a weighted blow-up, and
        \item There is a Cartier divisor $\cE\subseteq \cX$ such that $\cE\times_\cY Y$ is the exceptional divisor.
    \end{enumerate}
    Then $\cX\to \cY$ is a weighted blow-up.
\end{Prop}
In particular, if one has the Cartier divisor $\cE$, one can check that a morphism is a weighted blow-up locally on the target.
\begin{proof}
    There is an embedding $\cO_{\cX}(-m\cE)\to \cO_\cX$ given by the Cartier divisor $\cE$ for every $m\ge 0$,
    which induce embeddings $i_m:\pi_*\cO_{\cX}(-m\cE)\to \pi_*\cO_{\cX}$. The canonical map $\cO_\cY\to \pi_*\cO_\cX$ is an isomorphism. Indeed, it suffices to check that it is such locally on $Y$, and on $Y$ it is an isomorphism as $\pi_Y$ is a weighted blow-up.
    In particular, $i_m$ gives an ideal sheaf $I_m$ for every $m\ge 0$.
    
    Consider then $\cX'$ the weighted blow-up associated to the weighted embedding given by the $I_n$. We aim to show that $\cX\cong \cX'$ using the quotient description of a weighted blow-up of \Cref{def_wblowup_algebra_in_all_degrees}. In particular, it suffices to find a line bundle $\cL$ on $\cX$, and a $\Gm$-equivariant morphism $$\spec_{\cO_\cX}\left(\bigoplus_{n\in \bZ}\cL^{\otimes n}\right)\to \spec_{\cO_\cY}\left(\bigoplus_{n\in \bZ}I_n\right)$$
    where $I_n=\cO_\cY$ for $n\le 0$,
    which is an open embedding landing away from $V(\bigoplus_{n>0}I_n)$, the closed substack given by the vanishing of the ideal sheaf generated by $\bigoplus_{n>0}I_n$. A morphism as the one above is equivalent to a morphism as below over $\cY$
    $$\bigoplus_{n\in \bZ}I_n\to \bigoplus_{n\in \bZ}\pi_*(\cL^{\otimes n}).$$

    One can take $\cL:=\cO_{\cX}(-\cE)$, and the identity map $I_n = \pi_* \cO_{\cX}(-n\cE)\to \pi_* \cO_{\cX}(-n\cE)$ for every $n$. 
    Then it suffices to check the desired statement replacing $\cX$ with $X:=\cX\times_\cY Y$ and $\cY$ with $Y$. 
    In this case, let $J_\bullet$ be the sequence of ideals of the weighted blow-up $X\to Y$, and let $E$ be the exceptional divisor. As $\cO_X(-E)$ is the ideal sheaf corresponding to the natural inclusion $J_{\bullet + 1}\to J_\bullet$ (see \cite{QR}*{Def. 3.2.2}), one has that $\pi_*\cO_X(-nE)$ consists of the invariant elements, namely the elements of degree 0, so $\pi_*\cO_X(-nE) = J_n$. In particular $I_n = J_n$ for every $n\in \bZ$ as desired. 
\end{proof}

 \section{Reconstruction results for smooth Deligne-Mumford stacks}\label{section_prelims}
 %In what follows, all the stacks are assumed to be connected. Here is the main theorem of this section, which roughly speaking says that a smooth separated DM stack is determined by its coarse space and its stacky points of codimension up to one.

Fix a field $k$ and a smooth separated tame Deligne-Mumford stack $\mathcal{Y}$ over $k$. By \cite[Thm. 1]{GS17} if $\cY$ has trivial generic stabilizer, one may factor the coarse space morphism $\pi\colon \mathcal{Y} \to Y$ as follows
\[\mathcal{Y} \simeq \sqrt{Y^{\mathrm{can}}/(D,\mathbf{e})}^{\mathrm{can}} \to \sqrt{Y^{\mathrm{can}}/(D,\mathbf{e})} \to Y^{\mathrm{can}} \to Y \]
where $(-)^{\mathrm{can}}$ denotes the canonical stack associated to a stack with finite tame quotient singularities (see \cite[Sec. 2]{GS17}), $(D,\mathbf{e})$ denotes the ramification data of $\cY \to Y^{\mathrm{can}}$ and $\sqrt{Y^{\mathrm{can}}/(D,\mathbf{e})}$ denotes the associated root stack. Using this we will prove the following:

 %\label{rem:technicalcondition} Let $u\colon \mathscr{U} \to U$ be the coarse moduli map of a smooth tame Deligne-Mumford stack $\mathscr{U}$ and suppose we are given an open immersion  $U \subset X$ where $X$ is a separated $k$-scheme with finite quotient singularities and $U$ contains all codimension $\leq 1$ points of $X$. Let $D_i \subset U^{\mathrm{can}}$ denote the ramification divisors of $u$ with ramification degrees $e_i$. The $D_i$ extend uniquely to Cartier divisors $\mathcal{D}_i$ on the canonical stack $X^{\mathrm{can}}$. Assume that the root stack of $X^{\mathrm{can}}$ with respect to the $\mathcal{D}_i$ with degrees $e_i$

%In this section we prove that if $\cX$ is regular, it can be recovered from $X$ and an open substack $\cU \subset \cX$ containing all codimension $1$ points of $\cX$.
  \begin{Teo}\label{thm_DM_stack_is_determined_by_cms_and_codim1}
     Let $X$ be a separated $k$-scheme of finite type and suppose $U$ is a dense open subscheme of $X$ with complement of codimension $\geq 2$. Consider the coarse moduli map $\cU \xrightarrow{\mathrm{cms}} U$ where $\cU$ is a smooth tame Deligne-Mumford stack and let $(D,\mathbf{e})$ denote the (unique) extension of the ramification data of $\cU \to U^{\mathrm{can}}$ to $X^{\mathrm{can}}$, when the latter exists. If $X$ and $\sqrt{X^{\mathrm{can}}/(D,\mathbf{e})}$ both have finite tame quotient singularities, then there is a unique smooth tame Deligne-Mumford stack $\cX$ fitting into the cartesian square
      $$ 
    \begin{tikzcd}
    \cU \arrow[d] \arrow[r, dashed, hook, "\mathrm{open}"] & \cX \arrow[d, dashed, "\mathrm{cms}"] \\
    U \arrow[r, hook]   & X.                  
    \end{tikzcd}
    $$
     %such that its coarse moduli space is isomorphic to $X$, and with an open embedding $\cU\to \cX$ such that the compositions $\cU\to \cX\to X$ and $\cU\to U \to X$ agree.
 \end{Teo}

 \begin{example}
     The conclusion of Theorem \ref{thm_DM_stack_is_determined_by_cms_and_codim1} fails without the condition that $\sqrt{X^{\mathrm{can}}/(D,\mathbf{e})}$ has finite tame quotient singularities. Indeed, take $X=\mathbb{A}^3$ and the Cartier divisor $D=V(xy+z^2) \subset X$, and consider the root stack of degree $2$ of $U=X\smallsetminus\{(0,0,0)\}$ along $D^{\circ}=D \smallsetminus \{(0,0,0)\}$ and denote it by $\mathcal{U}=\sqrt{U/D^{\circ}}$. Here $\mathcal{U}$ is smooth because $D^{\circ}$ is smooth. Indeed, $\mathcal{U}$ is the stack quotient of $\spec (\mathcal{O}_U[T]/(T^n-(xy+z^2)))$ by $\bmu_n$ with the obvious action and the latter is smooth over $U$. However, there is no smooth separated Deligne-Mumford stack $\cX$ with coarse moduli space $X$ which is ramified at $D$ with degree $2$ (see \cite[Ex. 7]{GS17}). \end{example}
 %See also \cite[Remark 2 and Theorem 1]{GS17} for a similar statement. 

 \begin{Remark} \label{rem:etloccheck} Suppose $X$ is as in the first two sentences of the statement of Theorem \ref{thm_DM_stack_is_determined_by_cms_and_codim1}.  Since ramification data, the root stack and canonical stack constructions are all compatible with \'etale base change then the condition that $X$ and $\sqrt{X^{\mathrm{can}}/(D,\mathbf{e})}$ both have finite tame quotient singularities can be checked after replacing $X$ with an \'etale cover $X' \to X$. Moreover, it automatically holds for such an $X'$ if there is a smooth tame Deligne-Mumford stack $\cX'$ fitting into a cartesian diagram
 $$\xymatrix{\cU'\ar[r] \ar[d] & \cX' \ar[d] \\ U'\ar[r] & X'.}$$

\noindent where $U'=U \times_{X} X'$, $\cU'=\cU \times_{X} X'$, and both $\cU' \to U'$ and $\cX' \to X'$ are coarse moduli maps. 
 
 \end{Remark}
 
 We begin by analyzing the category of finite \'etale covers $\mathrm{Fet}_X$ over a stack $X$.
\begin{Lemma}\label{lemma_fet_equivalence}
    Let $Y$ be a regular Deligne-Mumford stack and $V\subseteq Y$ an open substack with complement of codimension at least 2. Then the functor 
    \[X \to Y \mapsto X \times_Y V \to V\]
    induces an equivalence of categories between $\mathrm{Fet}_Y$ and $ \mathrm{Fet}_V$.
\end{Lemma} 
When $Y$ is a scheme this is  \cite[Exp. 10, Cor. 3.3]{zbMATH01992950}. The general case follows by descent and we omit the proof.

%\begin{proof}
%Let $H$ be the normalization of $Y$ in the fraction fields of $H_V$. Then $H$ is normal, the map $H\to Y$ is  finite, and on codimension one it agrees with $H_V\to V$. Then it is unramified in codimension one, so it is \'etale from \cite[\href{https://stacks.math.columbia.edu/tag/0BMB}{Tag 0BMB}]{stacks-project}.
%The moreover part follows from the functoriality of the normalization.\Gio{details?}\A{I think it's fine}
%\end{proof}

\begin{Lemma} \label{lemma:trivialgerbeopen} Let $Y$ be a regular Deligne-Mumford stack and $V \subset Y$ an open substack with complement of codimension at least $2$. If $X \to Y$ is a separated Deligne-Mumford gerbe such that $X \times_Y V$ is trivial (i.e. admits a section $\sigma$), then $X \to Y$ admits a section which extends $\sigma$. \end{Lemma}

\begin{proof} The triviality of $X$ over $V$ corresponds to the existence of a $G_V$-torsor on $X \times_Y V$ for some finite \'etale group scheme $G_V$ over $V$ such that the associated map $X \times_Y V \to BG_V$ is an equivalence. By Lemma \ref{lemma_fet_equivalence}, $G_V$ extends uniquely to a finite \'etale group scheme $G$ over $X$ and this extends to a morphism of stacks $f\colon X \to BG$ which is an equivalence over $V$. However, morphisms between \'etale gerbes over $Y$ are equivalences when they are so on a dense subset of $Y$. Indeed, we can check this \'etale locally so we may assume $X=BG'$ where $G'$ and $G$ are constant. Thus $f$ \'etale locally corresponds to an element $\alpha 
\in \underline{\mathrm{Hom}}(G',G)(Y)$ (see \cite[Lem. 3.9]{AOV_tame}), and since $f$ is an equivalence on a dense open subset of $Y$, $\alpha$ must be an isomorphism. It follows that $f$ is an equivalence. \end{proof}

\begin{Lemma}\label{lemma_extension_gerbes}
    Let $Y$ be a regular Deligne-Mumford stack, $V\subseteq Y$ an open substack with complement of codimension at least 2, and $\cV\to V$ a gerbe banded by a finite constant group $G$ whose order is invertible in $k$. Then there is a gerbe $\cY\to Y$, banded by $G$, along with an equivalence of $G$-gerbes $\phi: \cY|_{V} \to \cV$ over $V$. 
    \end{Lemma}

    \begin{proof} Recall that if $G$ is a sheaf of groups on a Deligne-Mumford stack $S$, $\oH^2(S,G)$ parametrizes the isomorphism classes of $G$-gerbes on $S$. Thus when $G$ is a product of group schemes of the form $\bmu_n$, the lemma can be deduced from purity of Brauer group (see \cite[Thm. 1.1]{zbMATH07080116}). Indeed, the Kummer sequence and the fact that $\oH^2(Y,\mathbb{G}_m) \to \oH^2(V,\mathbb{G}_m)$ is an isomorphism (see, e.g., \cite[Prop. 2]{MR3521088} or \cite{344258}) yields the result. When $G$ is abelian, there is a finite \'etale cover $Y' \to Y$ such that $G|_{Y'}$ is a product of group schemes of the form $\bmu_n$. Now the result can be deduced for arbitrary abelian $G$ by considering the natural exact sequence
    \[0 \to G \to R_{Y'/Y}(G|_{Y'}) \to A \to 0\]
    where $R_{Y'/Y}(G|_{Y'})$ is the Weil restriction and $A$ is the cokernel. Note that $A$ is finite and \'etale. Indeed, compare the long exact sequences associated to $\oH^0(Y,-)$ and $\oH^0(V,-)$ to obtain:

    $$\xymatrix{ \oH^1(Y',G|_{Y'})\ar[d]\ar[r]&\oH^1(Y,A)\ar[d]\ar[r] & \oH^2(Y,G)\ar[d]\ar[r] &\oH^2(Y',G|_{Y'})\ar[d]\ar[r] & \oH^2(Y,A)\ar[d] \\
    \oH^1(V \times_Y Y',G|_{Y'})\ar[r]&\oH^1(V,A)\ar[r] & \oH^2(V,G)\ar[r] &\oH^2(V \times_Y Y',G|_{Y'}))\ar[r] & \oH^2(V,A) }$$

    Now Lemma \ref{lemma_fet_equivalence} implies the first two vertical arrows are isomorphisms, Lemma \ref{lemma:trivialgerbeopen} implies the last two vertical arrows are injective, and the $\bmu_n$ case implies the fourth arrow is an isomorphism. The isomorphism $\oH^2(Y,G) \to \oH^2(V,G)$ now follows from the five lemma. 
    
    For non-abelian $G$ with center $Z(G)$, we describe an action of $\oH^2(S,Z(G))$ on $\oH^2(S,G)$ that is functorial in $S$. Given $\cZ\to S$ and $\cG\to S$ a $Z(G)$ and $G$ gerbe respectively, the action of $\cZ$ sends $\cG$ to $(\cZ\times \cG)\sslash Z(G)$: the rigidification of $\cZ\times \cG$ by the diagonal subgroup $Z(G)$ of the inertia stack $I_{\cZ \times \cG}$. It follows from \cite[Thm. 3.3.3]{giraud} that this action is simply transitive, i.e., $\oH^2(S,G)$ is a $\oH^2(S,Z(G))$-torsor. In particular, there is a $Z(G)$-gerbe $\cZ_V\to V$ such that $\cV = (\cZ_V\times BG)\sslash Z(G).$ Now, since the result holds for abelian group schemes, we can extend the $Z(G)$-gerbe $\cZ_V\to V$ to a $Z(G)$-gerbe $\cZ\to Y$ and conclude by setting $\cY:=(\cZ\times BG)\sslash Z(G).$\end{proof}

Next we show that the extension of this gerbe is unique in the following sense: any two extensions are isomorphic and any two isomorphisms between them are uniquely $2$-isomorphic. 

    \begin{Lemma} \label{uniquenessGerbe} Let $Y$ be a regular Deligne-Mumford stack, $V \subset Y$ an open substack with complement of codimension at least 2 and suppose $\cY$ and $\cY'$ are Deligne-Mumford gerbes over $Y$ then:
    \begin{enumerate}
        \item any equivalence of stacks $f
        \colon \cY|_{V}\to \cY'|_{V}$ over $V$ extends to an equivalence $\cY \to \cY'$ over $Y$ and
        \item given two equivalences $\cY \xrightarrow{a,b}\cY' $ and a $2$-isomorphism $s\colon a|_V \to b|_V$, then there is a unique $2$-isomorphism $\sigma:a \to b$ extending $s$.
    \end{enumerate} \end{Lemma}
%\Mich{I am not sure about the statement. First of all, in part (2) it seems to me we can substitute $a \circ b$ with just a single morphism. Furthermore, I think we would like to prove that the two 2-categories are equivalent, but the statement seems to imply they are isomorphic. It is a bit weird because usually one proves first that the 2-morphisms are the same, and then that the 1-morphisms are the same up to 2-isomorphisms and so on. Here it seems that we are using (1) to prove (2). I am not sure but I think that the problem is that when we use the fact that a torsor extend uniquely it means that there exists a unique torsor up to non-unique isomorphism, and we need to be careful about this isomorphism.}

%\Sid{Which 2-categories are we saying are equivalent? I'm not sure I follow. Otherwise, I have taken your comments into consideration and made changes (now it has more of a "existence and then uniqueness" vibe) See if you like it better. I found it instructive to think about the example of the trivial gerbe B(Z/2).} 

\begin{proof}

For the first statement, consider the stack of equivalences between $\cY$ and $\cY'$
$$I=\underline{\Isom}_Y(\cY,\cY').$$ 
Every $y\in Y(\bar{k})$ admits an \'etale neighbourhood  $W\to Y$ such that  
$$I\times_Y W \cong \underline{\Isom}_W(W\times BG,W\times BG')$$
for finite constant group schemes $G$ and $G'$, and by \cite[Lem. 3.9]{AOV_tame} we have
$$\underline{\Isom}_W(W\times BG,W\times BG')\cong [\underline{\Isom}(G,G')/G']\times W.$$ 
where $G'$ acts by post-composing with conjugation. This shows that $I$ is a separated Deligne-Mumford stack. Moreover, the functor of isomorphism classes 
of $I$ is representable by a scheme $H$ finite \'etale over $Y$, because \'{e}tale locally over $Y$ the functor of isomorphism classes of $I$ is just the schematic quotient $(\underline{\Isom}(G,G')/G') \times W$, which is \'{e}tale over $W$. Therefore, $I\to H$ is a gerbe.

Now, since $f$ defines a section of $I\to Y$ over $V$, $H\to Y$ also has a section over $V$, and by \Cref{lemma_fet_equivalence} this yields a section $\sigma$ of $H\to Y$: indeed, we can regard $V\to H|_V$ as a morphism of finite \'{e}tale covers, hence the categorical equivalence yields an extension $Y\to H$, which must still be a section. Now, as the projection $Y\times_{\sigma, H} I \to Y$ has a section $f$ over $V$, by Lemma \ref{lemma:trivialgerbeopen}, the gerbe $Y\times_{\sigma, H} I \to Y$ admits a section extending $f$.  %\Sid{Aren't we done at this point? Let me know. It seems that the rest in this paragraph in not necessary but perhaps I'm missing something.} Namely, there is a finite group $\Gamma$ such that $Y\times_{\sigma, H}\underline{\Isom}_Y(\cY,\cY)\cong B\Gamma \times Y$. But then a section $Y\to B\Gamma\times Y$ of the projection is a $\Gamma$-torsor over $Y$. We have this $\Gamma$-torsor over $V$, so again by \Cref{lemma_fet_equivalence} the $\Gamma$-torsor extends uniquely.

The second part follows from Zariski's main theorem. Consider the following cartesian diagram
$$\xymatrix{\underline{\mathrm{Isom}}(a,b) \ar[d]_\pi \ar[rr] & & I \ar[d]^\Delta \\\ I\ar[rr]^-{(a,b)} && I\times_Y I}$$
Observe that $\pi$ has a section over $I \times_Y V \subset I$, so if we denote by $W$ the closure of its image in $\underline{\mathrm{Isom}}(a,b)$, then the corresponding map $W\to I$ yields a section. Indeed, it is a finite, representable and birational morphism with a normal target, therefore must be an isomorphism. 
\end{proof}

 \begin{Lemma}\label{cor_extend_gerbes}
     Let $X$ be a regular Deligne-Mumford stack, $U\hookrightarrow X$ an open substack with complement of codimension at least 2, and $\rho:\cU\to U$ a gerbe. Assume that $\cU$ is separated over $U$. Then there is a smooth and separated Deligne-Mumford stack $\cX\to X$ which is a gerbe, and which fits in a cartesian diagram as follows:
     $$\xymatrix{\cU\ar[r] \ar[d] & \cX\ar[d] \\ U\ar[r] & X.}$$
 \end{Lemma}
\begin{proof}
We may assume $X$ is irreducible. Since $\cU \to U$ is a separated Deligne-Mumford gerbe, for any point $p \in U$ the pullback of $\cU$ along the strictly local ring $R_p$ at $p$ in $U$ yields an isomorphism $\cU_{R_p} \cong BG_p$ for a finite constant group scheme $G_p$ over $R_p$. If $\eta \in U$ denotes the generic point then every $p \in U$ has a strictly local ring with a natural map $R_p \to R_{\eta}$, thus we see that the abstract group associated to $G_p$ doesn't depend on $p \in U$ and so we denote it by $G$. In particular, the band of $\cU$, which we will denote $\mathrm{Band}(\cU)$, is locally isomorphic to (the band associated to) $G$. Thus, the sheaf in the stack of bands $B_U=\underline{\mathrm{Isom}}_{\mathrm{Band}}(\mathrm{Band}(\cU),G)$ is a (left-)torsor under $\underline{\mathrm{Isom}}_{\mathrm{Band}}(G,G)=\operatorname{Out}(G)$.
Moreover, the $\operatorname{Out}(G)$-torsor $B_U\to U$ has the following property: the cartesian product $\cB_U:=B_U\times_U \cU$ is a gerbe over $B_U$ banded by $G$. 
Observe now that:
\begin{enumerate}
    \item the $\operatorname{Out}(G)$-torsor $B_U\to U$ extends to an $\operatorname{Out}(G)$-torsor $B_X\to X$ by \Cref{lemma_fet_equivalence}, 
    \item the banded $G$-gerbe $\cB_U\to B_U$ extends to a banded $G$-gerbe $\cB_X\to B_X$ by \Cref{lemma_extension_gerbes}, and
    \item the (strict) action, as in \cite{Rom05}*{Def. 1.3 (i)} of $\Out(G)$ on $\cB_U$ extends to $\cB_X$ by Lemma \ref{uniquenessGerbe}.
\end{enumerate}
Now we may set $\cX=[\cB_X/\Out(G)]$. It is a gerbe over $X$ since we may check this locally on $X$. Indeed, $\cX|_{B_X}\cong \cB_X$ and the latter is a gerbe over $B_X$. It also extends the given gerbe $\cU$ since, by \cite[Prop. 2.6]{Rom05}, the formation of the quotients commutes with \'etale base change.
\end{proof}

 \begin{proof}[Proof of \Cref{thm_DM_stack_is_determined_by_cms_and_codim1}] 
     By \cite[Ex. A.3]{AOV_tame}, we can factor the map $\cU \to U$ as $$\cU\xrightarrow{a} \cU_1 \xrightarrow{b} U$$
     where the map $a$ is a gerbe, and $\cU_1$ is a smooth tame Deligne-Mumford stack with generically trivial stabilizers.
         %\item the map $b$ is the canonical covering stack (see \cites{vistoli1989intersection, GS17}),
         %\item the map $c$ is a composition of root stacks, and
         %\item the map $d$ is a canonical covering stack.

     Using the unique extension of the ramification data of $\cU_1 \to U$ to $X^{\mathrm{can}}$ and the process outlined by \cite[Thm. 1.1]{GS17} there is a unique smooth tame Deligne-Mumford stack with generically trivial stabilizers $\cX_1$ equipped with a coarse space map
     $$\cX_1\xrightarrow{b'} X$$
     such that $b'$ extends $b$. The result now follows because we may extend the gerbe $\cU \to \cU_1$ uniquely to a gerbe $\cX \to \cX_1$ by \Cref{cor_extend_gerbes}.  \end{proof}
     
     %and the extensions are unique by the universal property of the canonical covering stack (see \cite[Theorem 4.6]{MR2774310})
     %First observe that these constructions over $U$ extend to $X$. Indeed, steps (2) and (4) extend by the construction of the canonical covering stack (see \cite[Proposition 2.8]{vistoli1989intersection}). Point (3) extends as the canonical covering stack is smooth, so any closed substack of pure codimension 1 is Cartier. Uniqueness follows from  Point (1) extends uniquely from \Cref{cor_extend_gerbes}.

 \section{Contraction on coarse moduli spaces}\label{section_contraction_on_cms}
\textcolor{black}{This section is divided into two subsections. First, we will present some result on deformation theory which will be useful for the rest of the manuscript. In the second half we will prove \Cref{teo_contraction_cms}.}
\subsection{Deformation theory and algebraization results}
\begin{Prop}\label{prop_extension_of_Y}
   \textcolor{black}{Let $f:\cE\to \cY$ be a proper morphism of Deligne-Mumford stacks, and let $\cE\hookrightarrow \cE'$ be a square-zero thickening of $\cE$ with ideal sheaf $I$. Assume that $R^1f_*I=0$ and that the natural map $\cO_{\cY}\to f_*\cO_\cE$ is an isomorphism. Then there is a square zero extension $\iota:\cY\hookrightarrow \cY'$ as below
   $$\xymatrix{\cE\ar[r] \ar[d]_f & \cE'\ar@{..>}[d]^{f'} \\ \cY\ar@{..>}[r] & \cY'}$$
   such that \begin{enumerate}
   \item the square above is a pushout in algebraic stacks,
        \item the ideal sheaf of $\cY$ in $\cY'$ is isomorphic to $\iota_*f_*I$, and
        \item $f'_*\cO_{\cE'}=\cO_{\cY'}$.
    \end{enumerate}
    Moreover, if $\cY$ is separated also $\cE'\to \cY'$ is proper.}
\end{Prop}
\textcolor{black}{Observe that we are \emph{not} claiming that the diagram of \Cref{prop_extension_of_Y} is cartesian,
and in fact it is not true in general.}
\begin{proof}%[Second proof of \Cref{prop_extension_of_Y}]
\textcolor{black}{ We formulate the desired problem in terms of extensions of cotangent complexes. All the maps will be over an algebraic stack $\cS$, so we will write $\bL_\cY$ for $\bL_{\cY/\cS}$, and similarly for $\cE$ and $\cE'$. From \cite{ols06}*{Thm. 1.1}\footnote{While the results in \cite{ols06} are stated for representable morphisms of algebraic stacks, one can check that they go through for morphisms of algebraic stacks which are representable by Deligne-Mumford stacks.},
%\cite[Thm. III.1.2.3]{illusie2006complexe},
the extension $\cE\to \cE'$ corresponds to an element $\alpha\in\E^1(\bL_{\cE}, I).$ There is an exact triangle 
     $f^*\bL_{\cY}\to \bL_{\cE}\to \bL_{\cE/\cY}\xrightarrow{+1}$
     which induces
     $$\E^1(\bL_{\cE},I)\to \E^1(f^*\bL_{\cY},I)\cong \E^1(\bL_{\cY},Rf_*I). $$}
     \textcolor{black}{We denote the composition by $g$.
        Moreover, there is an exact triangle
    $$f_*I\to Rf_*I\to \tau_{\ge 1} Rf_*I\xrightarrow{+1}$$}
    \textcolor{black}{inducing the maps
    $$\E^1(\bL_{\cY},  f_*I)\xrightarrow{h} \E^1(\bL_{\cY},Rf_*I)\to \E^1(\bL_{\cY},\tau_{\ge 1} Rf_*I)\xrightarrow{+1}.$$
     It follows from \cite[III.2.2.1]{illusie2006complexe}, using simplicial schemes as in \cite{ols06}, that one can find $\cY'$ as above if and only if $g(\alpha)$ belongs to the image of $h$. This, in particular, holds if $\E^1(\bL_{\cY},\tau_{\ge 1} Rf_*I)=0$, and in our case
     $$\E^1(\bL_{\cY},\tau_{\ge 1} Rf_*I) = \E^0(\cH^0(\bL_{\cY}),R^1f_*I) = 0$$
     where the first equality follows since the cotangent complex of a Deligne-Mumford stack is in non-positive degrees, and $\tau_{\ge 1} Rf_*I$ is in degrees greater than 0. The second equality instead follows from $R^1f_*I=0$.}

    \textcolor{black}{Observe that such a stack $\cY'$ is unique (up to unique isomorphism): in fact, any two solutions to the deformation problem above would differ by an element of $\E^0(\bL_{\cY},\tau_{\ge 1} Rf_*I)$, and this group is zero again because the cotangent complex is concentrated in degrees $\leq 0$ and the other complex is in degrees $\geq 1$.}
    
     \textcolor{black}{Recall also that the cotangent complex formalism fills the diagram above, together with any 2-isomorphism, in the following sense. Given a \textit{triple} $(f_1, f_2, \sigma)$ consisting of $f_1:\cY\to \cS$, $f_2:\cE'\to \cS$ and an isomorphism $\sigma:(f_1)|_\cE\to (f_2)|_\cE$, the cotangent complex formalism produces:
     \begin{enumerate}
         \item an extension $\cY'$ as above,
         \item a map $\cY'\to \cS$,
         \item an isomorphism $\Sigma:(\cE\to \cE'\to \cY')\to (\cE\to \cY\to \cY')$,
         \item an isomorphism between $\cY\to \cS$ and $\cY\to \cY'\to\cS$,
         \item an isomorphism between $\cE'\to \cS$ and $\cE'\to \cY'\to \cS$.
     \end{enumerate} By pulling-back the last two isomorphisms to $\cE$ and by composing them with $\Sigma$, they give $\sigma$. }
     
     \textcolor{black}{To prove that $f'_*\cO_{\cE'}=\cO_{\cY'}$, it suffices to push-forward the exact sequence $0\to j_*I\to \cO_{\cE'}\to j_*\cO_{\cE}\to 0$. From the vanishing of $R^1f_*I$, one has the following diagram:$$\xymatrix{0\ar[r] & \iota_*f_*I\ar[r] \ar[d]_\cong &\cO_{\cY'}\ar[r] \ar[d] & \iota_*\cO_{\cY}\ar[r] \ar[d]_\cong & 0 \\ 0\ar[r] & f'_*j_*I\ar[r]  &f'_* \cO_{\cE'}\ar[r] & f'_*j_*\cO_{\cE}\ar[r]  & 0}$$
     and the snake lemma gives the desired isomorphism.}
\end{proof}
\begin{Remark}\label{remark_contangent_complex_argument}
    \textcolor{black}{To construct $Y'$ we only used that $\E^0(\cH^0(\bL_{Y}),R^1f_*I) = \operatorname{Hom}(\Omega^1_{Y},R^1f_*I)=0$, rather than the stronger $R^1f_*I=0$. Moreover, we only used that $f_*\cO_\cE = \cO_\cY$ to guarantee that also $f'_*\cO_{\cE'} = \cO_{\cY'}$.}
\end{Remark}
\begin{Remark}
    \textcolor{black}{From the last paragraph of the proof of \Cref{prop_extension_of_Y} one can obtain a different proof of \Cref{prop_extension_of_Y} in the category of schemes. Indeed, one can simply construct $\cY'$, locally over an affine neighbourhood $\spec(A)\subseteq \cY$, as $\spec(\oH^0(\cE'|_{\spec(A)}))$.}
\end{Remark}
\begin{Cor}\label{cor_induction_for_thickenings}\textcolor{black}{Consider a diagram of Deligne-Mumford stacks as follows$$\xymatrix{\cE\ar[d]_f\ar[r]^\iota & \cX\\\cY & }$$let $I$ be the ideal sheaf of $\cE$ in $\cX$ and let $n\cE$ be the closed substack of $\cX$ with ideal sheaf $I^n$. Assume that $f$ is proper and that $\cY$ is separated.
Assume also that, if we denote by $I^{(n)}$ the ideal sheaf of $n\cE$ in $(n+1)\cE$, \begin{enumerate}
    \item the canonical map $\cO_Y\to f_*\cO_\cE$ is an isomorphism and \item$R^1f_*\iota^*I^{(n)}=0$ for every $n\ge 1$.
\end{enumerate} Then there is a commutative diagram as follows$$\xymatrix{\cE\ar[r] \ar[d]_f & n\cE\ar@{..>}[d]^{f^{(n)}}\ar[r] & \cX \\ \cY\ar@{..>}[r] & n\cY &}$$
such that the canonical map $\cO_{n\cY}\to f^{(n)}_*\cO_{n\cE}$ is an isomorphism. Moreover, the square above is a push-out in algebraic stacks.
}
\end{Cor}
\begin{proof}
    \textcolor{black}{The proof follows from repeatedly applying \Cref{prop_extension_of_Y}}.
\end{proof}

\begin{Prop}\label{prop_pushout_criteria}
    \textcolor{black}{Consider a commutative diagram of Deligne-Mumford stacks, where $\iota$ and $j$ are closed embeddings, $f$ is proper, $g$ is an isomorphism over $\cZ\smallsetminus \cY$, and $\cE\to \cY\times_\cZ\cX$ is a homeomorphism
     $$\xymatrix{\cE\ar[d]_f \ar[r]^\iota & \cX\ar[d]^g &\\ \cY\ar[r]^j&\cZ &}$$ 
     Let $I$ be the ideal sheaf of $\cE$ in $\cX$ and let $n\cE$ be the closed substack of $\cX$ with ideal sheaf $I^n$. 
Assume also that, if we denote by $I^{(n)}$ the ideal sheaf of $n\cE$ in $(n+1)\cE$, then \begin{enumerate}
    \item the canonical map $\cO_\cY\to f_*\cO_\cE$ is an isomorphism, \item$R^1f_*\iota^*I^{(n)}=0$ for every $n\ge 1$, and
    \item for every $p\in \cY$ there is an \'etale neighbourhood $\spec(A)$ of $p$ in $\cZ$ such that, if $J\subseteq A$ is the ideal of $\cY$ in $\cZ$, then $\lim_n(A/J^n)\cong \lim_n\oH^0(\cO_{n\cE\times_{\cZ}\spec(
A)}).$ 
\end{enumerate}
 Then the square above is a push-out square in the category of algebraic stacks. 
}
\end{Prop}
\begin{Remark}
    In the proof of \Cref{prop_pushout_criteria} will use our assumptions on all the stacks being essentially of finite type and with affine stabilizers, to apply \cites{mayer_viet_squares,HR19}. Moreover, condition (3) in \Cref{prop_pushout_criteria} always holds if all stacks above are algebraic spaces and if $\cZ$ is constructed using \cite{artin1970algebraization}, see \cite{artin1970algebraization}*{Lem. 6.5}.
\end{Remark}
\begin{proof}
\textcolor{black}{We first focus on the case in which $\cZ=Z=\spec(A)$ as in point (3), and so $\cY$ is a scheme which we denote by $Y$. Let $\widehat{A}:=\lim_n(A/J^n)$, let $\widehat{Z}:=\spec(\widehat A)$ and let $\cS$ be an algebraic stack with maps $\cX\to \cS$ and $Y\to \cS$, and an isomorphism $(\cE\to Y\to \cS)\to (\cE\to \cX\to \cS)$.
}

\textcolor{black}{First we prove that there is a morphism $\widehat{Z}\to \cS$ extending $Y=\spec(A/J)\to \cS$.
    Indeed, from Tannaka duality \cite[Thm. 1.1]{HR19}, we have an equivalence of groupoids 
 $$\Hom(\widehat Z,\cS) \cong \Hom( {\rm Coh}(\cS), {\rm Coh}(\widehat{Z})).$$
 Since the pair  $(\spec(A/J), \widehat{Z})$ is coherently complete, we have that 
 $$ \Hom({\rm Coh}(\cS), {\rm Coh}(\widehat{Z})) \cong  \Hom({\rm Coh}(\cS), \lim_n {\rm Coh}(\spec(\oH^0(n\cE)))) \cong \lim_n \Hom({\rm Coh}(\cS), {\rm Coh}(\spec(\oH^0(n\cE))))  $$
 which using Tannaka duality again gives us 
 $$ \Hom(\widehat{Z}, \cS) \cong \lim_n \Hom(\spec(\oH^0(n\cE)), \cS).$$
Therefore it is enough to find a compatible sequence of morphisms $\spec(\oH^0(n\cE))\to \cS$. This follows from \Cref{cor_induction_for_thickenings}.}

\textcolor{black}{Therefore we have the following diagram:
$$\xymatrix{\cE\ar[d] \ar[r] & \widehat{Z}\times_Z\cX\ar[d] \ar[r] & \cX\ar[d] \\ Y\ar[r]& \widehat{Z}\ar[r] &Z}$$and we have maps $\cX\to \cS$ and $\widehat{Z}\to \cS$. We need to check that they agree. This follows again from Tannaka duality. Indeed, the pair $(\widehat{Z}\times_Z\cX,\cE)$ is coherently complete, so $$\Hom(\widehat{Z}\times_Z \cX, \cS) \cong \lim_n \Hom(\spec(A/I^n)\times_Z\cX, \cS) \cong \Hom(n\cE, \cS)$$where the last equivalence follows since the topological spaces of $\cE$ and $\spec(A/I)\times_Z\cX$ agree, so if $\cJ_n$ are the ideal sheaves of $\spec(A/I^n)\times_Z\cX$, then the sequences of ideals $\{\cJ_n\}$ and $\{\cO_\cX(n\cE)\}$ are such that for every $m\ge 0$ there is an $n\ge 0$ such that $\cJ_n\subseteq \cO_\cX(-m\cE)$ and $\cO_\cX(-n\cE)\subseteq \cJ_m$.}

\textcolor{black}{Therefore we can safely add the arrows to $\cS$, with no ambiguity:$$\xymatrix{\cE\ar[d] \ar[r] & \widehat{Z}\times_Z\cX\ar[d] \ar[r] & \cX\ar[d]\ar[dr]& \\ Y\ar[r]& \widehat{Z}\ar[r] \ar@/_1pc/[rr] &Z &\cS}$$ But now observe that $\widehat{Z}\to Z$ is flat, it induces an isomorphism on $Y$, and the map $\cX\to Z$ is an isomorphism on $Z\smallsetminus Y$. Therefore in the following diagram, both the left square and the outer square are pushouts from \cite{mayer_viet_squares}*{Thm. A}:
$$\xymatrix{\widehat{Z}\times_Z(Z\smallsetminus Y) \ar[r] \ar[d] & \widehat{Z}\times_Z\cX\ar[d]\ar[r] & \widehat{Z}\ar[d]\\ Z\smallsetminus Y = \cX\smallsetminus \cE\ar[r] & \cX \ar[r] & Z}$$
So also the right square is a pushout.}

\textcolor{black}{The general case when $\cZ$ is a Deligne-Mumford stack follows from \'etale descent (to construct the map $\cZ\to \cS$) and the fact that the group of isomorphisms form a sheaf, and this enables us to construct the two isomorphisms $(\cX\to \cS)\to (\cX\to \cZ\to \cS)$ and $(\cY\to \cS)\to (\cY\to \cZ\to \cS)$.}
\end{proof}

\subsection{Proof of \Cref{teo_contraction_cms}}
The goal of this section is to prove the following result:
\begin{Teo}\label{teo_contraction_cms}
    Let $\cY$ be a separated, smooth \textcolor{black}{tame} Deligne-Mumford stack and let $\pi:\cE \rightarrow \cY$ be a weighted projective bundle with positive dimensional fibers. Suppose we are given a closed embedding $\cE \hookrightarrow \cX$ of codimension $1$ (i.e. $\cE$ is a Cartier divisor in $\cX$) where $\cX$ is a smooth \textcolor{black}{tame} and separated Deligne-Mumford stack. Furthermore, assume that the normal bundle $\cN_{\cE\vert\cX}$ is isomorphic to $\cO_{\cE}(-1)\otimes \pi^*\cL$ for a line bundle $\cL$ on $\cY$. Then there exists a contraction
    $$ 
    \begin{tikzcd}
    E \arrow[d, "p"] \arrow[r, hook] & X \arrow[d, dashed, "\rho"] \\
    Y \arrow[r, dashed, hook]   & Z                  
    \end{tikzcd}
    $$
    at the level of coarse moduli spaces such that:
    \begin{enumerate}
        \item $\rho$ is an isomorphism away from $E$,
        \item $\rho$ is projective,
        \item $\rho_*\cO_X \simeq \cO_Z$,
        \item there is a highly divisible integer $\ell$, an isomorphism $\cO_Z^{\wedge} \simeq  \varprojlim_k\left(p_*\cO_{k\ell E} \right)$ where the completion of $\cO_Z$ is with respect to the ideal sheaf of $Y$,
             $k\ell E$ is the coarse space of the closed substack with ideal sheaf $\cO_{\cX}(-k \ell \cE)$, and
             $p$ is as in the diagram above.
        \item $Z$ is a separated normal algebraic space with $Y\to Z$ a closed embedding.        
    \end{enumerate} Moreover, given any other algebraic space $Z'$ and morphisms $X\to Z'$, $Y\hookrightarrow Z$ with these properties, we have an isomorphism $Z\simeq Z'$ which commutes with all the other morphisms.
\end{Teo}

\begin{Lemma}\label{lemma:step1}
    With the same hypotheses and notation of \Cref{teo_contraction_cms}, let $n\cE$ denote the closed substack of $\cX$ defined by the ideal $I^n$, where $I$ denotes the ideal sheaf defining $\cE$; set $1\cY:=\cY$ and $\pi^{(1)}=\pi$. Then there exist unique (up to unique isomorphism) algebraic stacks $\{n\cY\}_{n\geq 1}$, infinitesimal thickenings by square-zero ideals $n\cY\hookrightarrow(n+1)\cY$ and morphisms $\pi^{(n)}:n\cE\to n\cY$ such that for each $n$ there is a cartesian diagram
    $$ 
    \begin{tikzcd}
    n\cE \arrow[d] \arrow[r, hook] & (n+1)\cE \arrow[d] \\
    n\cY \arrow[r, hook]   & (n+1)\cY.                  
    \end{tikzcd}
    $$
    Moreover, we have a canonical isomorphism $\pi^{(n)}_*\cO_{n\cE} \simeq n\cY$.
\end{Lemma}

\begin{proof}\textcolor{black}{It suffices to check that the assumptions of \Cref{cor_induction_for_thickenings} apply. The only non-trivial statement to check is that $R^1\pi_*\iota^*I^{(n)}=0$, where $\iota:\cE\to (n+1)\cE$ is the inclusion of the reduced structure.   
    But $\iota^*I^{(n)} =N_{\cE/\cX}^{\otimes -n} $ since $\cE$ is Cartier, and $\pi$ is flat, the fibers of $\pi$ are weighted projective bundles, and $\cN_{\cE\vert\cX}^{\otimes -n}$ restricts to $\cO_\cP(n)$, where $\cP $ is a weighted projective stack. The desired vanishing follows as $\cO_\cP(n)$ has no higher cohomology from \cite[Prop. 2.5]{meier2015vector}, and from cohomology and base change.}
\end{proof}

\begin{Prop}\label{prop:existence contraction cms}
  With the same hypotheses and notation of \Cref{teo_contraction_cms},  there exists a contraction 
    $$ 
    \begin{tikzcd}
    E \arrow[d] \arrow[r, hook] & X \arrow[d, dashed, "\rho"] \\
    Y \arrow[r, dashed, hook]   & Z                  
    \end{tikzcd}
    $$
   at the level of coarse moduli spaces which is an isomorphism away from $E$ and such that $X \rightarrow Z$ is a proper morphism and $Y\to Z$ is a closed embedding. Moreover, we have $\rho_*\cO_X \simeq \cO_Z$ and $\cO_Z^{\wedge} \simeq  \varprojlim_k\left(p_*\cO_{k\ell E} \right)$, where the completion of $\cO_Z$ is with respect to the ideal sheaf of $Y$.
    \end{Prop}
    \begin{proof}
    First let us recall the following: if $\cX\to\cY$ is a closed embedding of tame Deligne-Mumford stacks, then the corresponding map of coarse moduli spaces $X\to Y$ is a closed embedding. Moreover if $\cI_\cX\to \cO_\cY$ is the ideal sheaf of $\cX$ and $\pi:\cY\to Y$ is the coarse moduli space map, then $\pi_*\cI_\cX\to \pi_*\cO_\cY = \cO_Y$ is the ideal sheaf of $X$, because $\cY$ is a tame Deligne-Mumford stack.

    As in the proof of \Cref{lemma:step1}, we denote $\cI=\cO_{\cX}(-\cE)$ the ideal sheaf of $\cE$, for a positive integer $n$ we let $n\cE$ be the closed substack with ideal sheaf $\cI^n=\cO_{\cX}(-n\cE)$ and we denote the square-zero sheaf of ideals of $n\cE$ in $n\cE$ by $J^{(n)}:=i_*N_{\cE/\cX}^{\otimes -n}$, where $i:\cE\to (n+1)\cE$ is the inclusion of the reduced structure. Recall that $R^1\pi^{(n)}_*J^{(n)}=0$.   
    
     Take a positive integer $\ell$ which is divisible by the order of every stabilizer of a closed point of $\cX$. Then $\cO_\cX(-\ell\cE)$ descends to a line bundle $I$ on $X$, the coarse moduli space of $\cX$ (since the stabilizers act trivially on $\cO_\cX(-\ell\cE)$). If we denote by $\ell E$ the coarse moduli space of $\cE^{(\ell)}$, then $\ell E\to X$ is a closed embedding with ideal sheaf $I$. So in particular $\ell E$ is a Cartier divisor, and we have a diagram as follows, where $\ell Y$ is the coarse moduli space of the algebraic stack constructed using \Cref{cor_induction_for_thickenings}:
     $$\xymatrix{\ell E\ar[r] \ar[d]_{p^{(\ell )}} & X \\ \ell Y. &}$$
     According to \cite[Cor. 6.10]{artin1970algebraization}, it suffices to prove that
     \begin{enumerate}
         \item for every coherent sheaf $F$ on $\ell E$ we have $R^1p^{(\ell )}_*(F\otimes(I/I^2)^{\otimes m}) = 0$ for $m$ big enough, and
         \item $R^1p^{(\ell )}_*((I/I^2)^{\otimes m}) = 0$ for every $m>0$.
     \end{enumerate}

     Observe that, if $\iota:E\to \ell E$ is the inclusion of the reduced structure, $\iota^*I/I^2$ restricts to $\cO_P(\ell )$ along each fiber of $E\to Y$, where we denote by $\cO_P(\ell )$ the push forward of $\cO_{\cP}(\ell )$ via the coarse moduli space map $\cP\to P$, and $\cP$ is a weighted projective stack. In particular, $\iota^*I/I^2$ is relatively ample. So from \cite[\href{https://stacks.math.columbia.edu/tag/0B5V}{Tag 0B5V}]{stacks-project} also $I/I^2$ is relatively ample, hence the first vanishing is just Serre's vanishing theorem.

     For the second vanishing it suffices to check that, if $\cL$ is a line bundle on $\ell \cE$ such that $i^*\cL$ (which is a line bundle on $\cE$) restricts to $\cO_\cP(k)$ along each fiber $\cP$ of $\cE\to \cY$, with $k>0$, then $R^1\pi^{(\ell)}_*(\cL^{\otimes m}) = 0$ for every $m>0$.

    On $\cE$, the desired vanishing follows from cohomology and base change, as $\cO_{\cP}(k)$ has no higher cohomology when $k>0$, since $\cP$ is a weighted projective bundle. On any thickening $\cE \hookrightarrow n\cE$ this follows by an inductive argument. Thus we assume that the desired vanishing holds for $n$ and we will show that it holds for $n+1$. Consider the short exact sequence
     $$0\to J^{(n)}\to \cO_{(n+1)\cE}\to \cO_{n\cE}\to 0.$$
     This leads to 
     $$0\to J^{(n)}\otimes \cL^{\otimes m}\to \cL^{\otimes m}\to \cL^{\otimes m}|_{n\cE}\to 0.$$
     Observe first that $$R^1\pi^{(n+1)}_*(\cL^{\otimes m}|_{n\cE})=j^{(n)}_*R^1\pi^{(n)}_*(\cL^{\otimes m}|_{n\cE})$$ where $j^{(n)}$ is the closed embedding $n\cY\to (n+1)\cY$, (because the push forward via a closed embedding is exact). By induction $R^1\pi^{(n)}_*(\cL^{\otimes m}|_{n\cE})=0$, so $R^1\pi^{(n+1)}_*(\cL^{\otimes m}|_{n\cE})=0$, and it suffices to check that $R^1\pi_*^{(n+1)}(J^{(n)}\otimes \cL^{\otimes m})=0$ as well. But $J^{(n)}\otimes \cL^{\otimes m} = i_*(N_{\cE/\cX}^{\otimes -n}\otimes i^*\cL^{\otimes m})$ so
     $$R^1\pi_*^{(n+1)}(J^{(n)}\otimes \cL^{\otimes m}) = j_*R^1\pi_*(J^{(n)}\otimes \cL^{\otimes m})$$ which is 0 as $R^1\pi_*(J^{(n)}\otimes \cL^{\otimes m})=0$. Therefore, we may apply \cite[Cor. 6.10]{artin1970algebraization} to obtain a $Z$ as in the statement of \Cref{prop:existence contraction cms} and we are left with proving the moreover part.

     Let $\cO_Z^{\wedge}$ denote the formal completion of $\cO_Z$ along $\ell Y$. We claim that the natural map $\cO_Z \to \rho_*\cO_X$ is an isomorphism upon completion: $\cO_Z^\wedge\simeq (\rho_*\cO_X)^\wedge$. If this is the case, then by \cite[Thm. 2.2]{artin1970algebraization} we obtain $\cO_Z \simeq \rho_*\cO_X$, because we already have an isomorphism $\cO_{Z\smallsetminus nY} \simeq \rho_*\cO_{X\smallsetminus nE}$.

     To prove our claim, first note that $\cO_Z^{\wedge}$ is isomorphic to the formal algebraic space $\mathfrak{X}$ constructed in \cite[Thm. 3.1, Thm. 6.2]{artin1970algebraization}. Moreover, (see \cite[Lem. 6.5]{artin1970algebraization}) we have 
     \[\cO_Z^{\wedge} \simeq \varprojlim_k \left(p^{(\ell)}_*\cO_{k\ell E}\times_{p_*^{(\ell)}\cO_{\ell E}} \cO_{\ell Y} \right) \simeq  \varprojlim_k\left(p^{(\ell)}_*\cO_{k\ell E} \right),\]
     where in the last isomorphism we used the fact that $p_*^{(\ell)}\cO_{\ell E} \simeq \cO_{\ell Y}$: this follows from the isomorphism $\pi^{(\ell)}_*\cO_{\ell\cE} \simeq \cO_{\ell\cY}$ of \Cref{lemma:step1}, pushed forward along the coarse moduli space map $\ell\cY\to \ell Y$.

     Observe now that by the theorem on formal functions \cite{EGAIII}*{4.1.5} we have $\varprojlim\left(p^{(\ell)}_*\cO_{k\ell E} \right) \simeq (\rho_*\cO_X)^\wedge$ (note that this agrees with the limit in \Cref{prop:existence contraction cms} since $E$ and $\ell E$ have the same underlying topological space). Putting all together, we deduce that $\cO_Z^\wedge\simeq (\rho_*\cO_X)^\wedge$, as claimed.
    \end{proof}

    \begin{Prop}
        The algebraic space $Z$ constructed in \Cref{prop:existence contraction cms} is normal.
    \end{Prop}
\begin{proof}
    To prove that $Z$ is normal we use Serre's criterion for normality \cite{EGAIV}*{5.8.6}. For this we have to prove that
    \begin{enumerate}
    \item[($R_1$)] every codimension one point of $Z$ is regular, and
    \item[($S_2$)] for every open subset $U\subset Z$ and every closed subset $W$ of codimension $\geq 2$, every regular function on $U\smallsetminus W$ extends to $U$.
    \end{enumerate}
    In what follows, we denote $\rho:X\to Z$ the morphism constructed in \Cref{prop:existence contraction cms}.
    The condition ($R_1$) is easy to check: any codimension one point $\zeta$ of $Z$ lies in $Z\smallsetminus Y$, hence the modification $\rho:X\to Z$ is an isomorphism around $\zeta$. Let $\xi$ be the unique point in the preimage of $\zeta$: then $\cO_{X,\xi}\simeq\cO_{Z,\zeta}$. Observe that $X$ is normal, because it is the coarse moduli space of a smooth Deligne-Mumford stack: in particular, by Serre's criterion \cite{EGAIV}*{5.8.6}, we have that $\cO_{X,\xi}$ is regular, hence also $\cO_{Z,\zeta}$ is regular.

    To check ($S_2$), let $U\subset X$ be an open subset and $W\subset U$ be a closed subvariety of codimension $\geq 2$. Let $\varphi$ be a regular function on $U\smallsetminus W$. 
    
    Suppose first that $W$ does not contain $Y\cap U$: in this case $\rho^*\varphi$ is a regular function on $\rho^{-1}U \smallsetminus \rho^{-1}W$ and $\rho^{-1}W$ has codimension $\geq 2$, hence by normality of $X$ we can extend $\rho^*\varphi$ to a regular function $\psi'$ on $\rho^{-1}U$. 

    Suppose instead that $W$ contains $Y\cap U$, hence its preimage along $\rho$ consists of the union of the proper transform $\widetilde{W}$ of $W$ and $E\cap \rho^{-1}U$. We can extend $\rho^*\varphi$ to a function on $\rho^{-1}U\smallsetminus (E\cap \rho^{-1}U)$, and we can further extend it to a section of $\cO(nE)(\rho^{-1}U)$ for some integer $n$. 
    
    If $n\leq 0$, then we have an inclusion $\cO(nE)\hookrightarrow \cO_X$, i.e. $\rho^*\varphi$ actually extends to a regular function $\psi'$ on $\rho^{-1}U$. If $n>0$ then we can look at the short exact sequence
    \[0 \to \cO((n-1)E) \to \cO(nE) \to \cO(nE)|_{E} \to 0. \]
    As the last term is isomorphic (up to shrinking the base) to $\cO_E(-n)$, by applying the functor $\Gamma(\rho^{-1}U,-)$ to this sequence we get an isomorphism $\Gamma(\rho^{-1}U,\cO(nE))\simeq \Gamma(\rho^{-1}U,\cO((n-1)E))$, hence we can actually extend $\rho^*\varphi$ to a section of $\cO((n-1)E)(\rho^{-1}U)$. Repeating this argument, we eventually deduce that $\rho^*\varphi$ extends also in this case to $\psi'\in \cO_X(\rho^{-1}U)$.

    By \Cref{prop:existence contraction cms}, we have that $\rho^{\sharp}:\cO_Z \to \rho_*\cO_X$ is an isomorphism, hence we can take the preimage of $\psi'$ along $\rho^\sharp$, which by construction will be an extension of $\varphi$ to $U$, thus concluding the proof.
\end{proof}
For what follows, we need the following technical result.

\begin{Lemma}\label{lemma_isom_coarse_moduli_space}
    Let $X\xrightarrow{p} X'$ be a proper morphism between normal separated algebraic spaces. Assume that $E\subseteq X$ is a pure codimension one irreducible closed algebraic space (namely, a divisor) with a map $p_E:E\to Y'$. Assume also that $p$ is an isomorphism away from the image of $E$, and the map $E\to X\to X'$ factors as $E\to Y'$ and $Y'\to X'$, with $Y'\to X'$ finite, as in the following cartesian diagram:
    \[
    \xymatrix{X\smallsetminus E \ar@{^{(}->}[r]\ar[d]_\cong & X \ar[d]^p \\ X'\smallsetminus p(E)\ar@{^{(}->}[r]& X'}
    \]
    Then $X'$ is uniquely determined by $p_E$ and $X\smallsetminus E$. 
\end{Lemma}
\begin{proof}
    Let $X'_1$ and $X'_2$ be two such algebraic spaces, with $p_i:X\to X'_i$; consider $X\to X_1'\times X_2'$ the product map $(p_1, p_2)$, let $\Gamma\subseteq X_1'\times X_2'$
    be its image, and $\Gamma^n$ its normalization. Since $X\to X_i'$ is proper, also $X\to X_1'\times X_2'$ is proper from \cite[\href{https://stacks.math.columbia.edu/tag/04NX}{Tag 04NX}]{stacks-project}. Similarly, $\Gamma\to X'_i$ is proper from \cite[\href{https://stacks.math.columbia.edu/tag/0AGD}{Tag 0AGD}]{stacks-project}. Since the normalization is a finite (hence proper) map, $\Gamma^n\to X'_1$ is proper.
    
    We claim that the map $\Gamma\to X'_1$ is quasi-finite. {The fiber of a point $x$ along this map is in natural bijection with $p_2(p_1^{-1}(x))$. Clearly, the fiber over a point in $X_1'\smallsetminus p_1(E)$ is just a single point. Let $x$ be a point in  $p_1(E)$. The restriction $p_i|_E\colon E \to X'_i$ factors as $E \to Y' \overset{f_i}{\to} X'_i$. Therefore, the fiber of $x$ along $\Gamma \to X'_1$ is in natural bijection with $f_2(f_1^{-1}(x))$, which is a finite set as both $f_1$ and $f_2$ are finite.}
    
    %The map $(p_1,p_2)\colon X\to X_1'\times X_2'$ factors through $f:X\to \Gamma$, and the fiber of $x$ along $\Gamma\to X'_1$ is $f(p^{-1}_1(x))$ (observe that $f$ is surjective as it is dominant and proper). But $f$ contracts $E$, namely, the composition $E\to X\xrightarrow{f} \Gamma$ factors as $E\to Y'\to \Gamma$, because $p_i|_E$ factors through $Y'$ for $i=1,2$. We deduce that the fiber of $x$ along $\Gamma \to X'_1$ is isomorphic to the image through $Y'\to X'_2$ of the preimage of $x$ along $Y'\to X_1'$. As the latter is finite, we have proved our claim.

    As the normalization is finite, also $\Gamma^n\to \Gamma\to X'_1$ is a proper, quasi-finite and birational morphism. Then it is an isomorphism from Zariski's main theorem. So the rational map $X'_1\dashrightarrow X_2'$ is a morphism. As the corresponding map of sets is quasi-finite and they are proper and birational, again from Zariski's main theorem they are isomorphic.
\end{proof}
We have now collected all the ingredients necessary to prove the main result of this section.
\begin{proof}[Proof of \Cref{teo_contraction_cms}]
    The existence of the contraction $X\to Z$ is \Cref{prop:existence contraction cms}. For proving the uniqueness part of the statement, suppose that $Z'$ is another scheme satisfying the thesis of the theorem: then we can apply \Cref{lemma_isom_coarse_moduli_space} to conclude that $Z'\simeq Z$. \textcolor{black}{Finally, the map $X\to Z$ is projective. Indeed, there is an integer $k>0$ such that $\cO_X(kE)=\cL$ is a Cartier divisor, and $\cL^{-1}$ is ample over $Z$, since the positive dimensional fibers of $X\to Z$ are weighted projective schemes, and $\cL$ restricts to an ample divisor on those.}
\end{proof}

 \section{The case of affine $\cY$}\label{section_local_analisis}
 In this section we prove \Cref{thm_affine_contraction}, which itself forms the heart of the proof of \Cref{Thm_intro}. To begin we establish some lemmas. 

 \begin{Lemma} \label{lem:surjobvi} Fix a finite module $V$ over a Noetherian ring $B$ and let $Y=\spec(A)$ be a finite scheme over $\spec(B)$. If we have a surjective map of $B$-modules $\rho\colon V \to A$, then the map $\spec(A) \to \mathbb{P}_B(V)$ induced by the surjection $V \otimes_B A \to A$ is a closed immersion. 
 \end{Lemma}

 \begin{proof} It suffices to check this fiber locally on $\spec(B)$, so we may assume that $B$ is a field, $V$ is a vector space, and $A$ a Artin local ring. Now choose a $v \in V$ such that $v \mapsto 1 \in A$, then the induced map $\spec(A) \to \mathbb{P}(V)$ factors through the non-vanishing of $v \in \oH^0(\mathbb{P}(V),\mathcal{O}(1))$ and the associated map of coordinate rings is surjective because it contains the image of $\rho$.\end{proof} 

\begin{Teo}\label{thm_affine_contraction} Consider two affine schemes with a finite morphism $\spec(A)\to \cB:=\spec(B)$, with $\spec(A)$ smooth and connected and $\xi\in \spec(A)$ a point.
Assume that $\cX$ is a \textcolor{black}{tame} Deligne-Mumford stack, separated over $\cB$, and with coarse moduli space $X$ projective over $\cB$. Let $\cE\subseteq \cX$ be a Cartier divisor in $\cX$ isomorphic to $\cP:=\cP_A(a_1,...,a_n)$, with $0<a_1\le ...\le a_n$. Assume that $\cN_{\cE\vert\cX}:=\cO_\cX(\cE)|_\cE\cong \cO_{\cP}(-1)$.

Then, up to replacing $\cB$ with a Zariski open subset of $\cU\subseteq \cB$ containing the image of $\xi$, and $\spec(A)$, $\cE$ and $\cX$ with $\spec(A)\times_\cB \cU$, $\cE\times_\cB \cU$ and $\cX\times_\cB \cU$ respectively, there is an $\cX'$ fitting into a commutative diagram
 $$ 
    \begin{tikzcd}
    \cE \arrow[d] \arrow[r, hook] & \cX \arrow[d, dashed] \\
    \spec(A) \arrow[r, dashed, "i", hook]   & \cX'                  
    \end{tikzcd}
    $$
    such that
\begin{enumerate}
    \item $\cX'$ is a separated \textcolor{black}{tame} Deligne-Mumford stack with coarse moduli space projective over $\cB$,
    \item $i$ is a closed immersion which realizes $\spec(A)$ as a complete intersection of codimension $n$ and $\cX'$ is smooth in a neighbourhood of $\spec(A)$, and 
    \item there is an isomorphism between $\cX$ and a weighted blow-up of $\cX'$ at $\spec(A)$
    with weights $a_1,...,a_n$, which identifies $\cE$ with the exceptional divisor of such a blow-up.
\end{enumerate}
In particular, $\cX'$ is a scheme in a neighbourhood of $\spec(A)$, and $\cX\to \cX'$ is an isomorphism away from $\cE$.
\end{Teo}
We now report two remarks which we will use in the proof of \Cref{thm_affine_contraction}.
\begin{Remark}\label{rmk_sm_cartier_on_smooth}
 If $X$ is a finite-type $k$-scheme where $D$ is a Cartier divisor of $X$ and $p\in D$ a smooth point of $D$, then $p$ is a smooth point of $X$. 
 \end{Remark}
\begin{Remark}\label{remark_cX_trivial_gen_stab}
Fix a separated algebraic stack $\cX$ and let $\cO_\cX(\cD)$ be a line bundle associated to a Cartier divisor $\cD$ on $\cX$. If $\cO_\cX(\cD)|_\cD$ induces a representable morphism $\cD\to \cB\Gm$, then there is a open substack $\cV \subset \cX$ containing $\cD$ such that $\cV\smallsetminus \cD$ has trivial stabilizers.
Indeed, being representable is an open condition for separated stacks, so there is an open substack $\cV\subseteq \cX$ such that $\cO_\cV(\cD)$ induces a representable morphism $\cV\to \cB\Gm$. As $\cO_{\cX}(\cD)\cong \cO_\cX$ away from $\cD$, the stack $\cV$ is an algebraic space away from $\cD$. In particular, $\cX$ has generically trivial stabilizer. 
\end{Remark}

\begin{proof}
Since $\cO_{\cP}(-1)$ induces a representable morphism $\cE\to B\Gm$, by \Cref{remark_cX_trivial_gen_stab} the stack $\cX$ has trivial stabilizers on $\cV \setminus \cE$ where $\cV$ is an open substack containing $\cE$. This implies the stacky structure of $\cX$ lies on the union of $\cE$ and the closed substack $\cX \setminus \cV$. As such, to prove the desired result it suffices to:

\begin{enumerate}
\item construct the coarse moduli space $X'$ of $\cX'$,
\item show there is a closed immersion $\spec(A) \to X'$,
    \item prove the coarse moduli space $X'$ is smooth around $\spec(A)$,
    \item show $X'$ contains $\spec A$ as a complete intersection of codimension $n$, and
    \item show that, in a neighbourhood of $\cE$, the stack $\cX$ is a weighted blow-up of $X'$ along $\spec(A)$.
\end{enumerate} 

\noindent Indeed, the remaining stacky structure on $\cX'$ is necessarily disjoint from an open substack containing $\spec(A)$, and therefore determined by $\cX$.

To do so, we will explicitly construct the contraction $X\to X'$ of \Cref{teo_contraction_cms} that contracts $E$ (the coarse moduli space of $\cE$). In this process, we will produce Cartier divisors $\cD_1$,..., $\cD_n$ on $\cX$ and their images $\widetilde{D}_1,$ ..., $\widetilde{D}_n$ in $X'$ such that $\cD_i \cap \cP_A(a_1,...,a_n)=\cP_A(a_1,..,\widehat{a_i},..,a_n)$; this enables the use of induction to establish (1)-(4) as in Koll\'ar's proof of Castelnuovo's theorem in \cite{kol_resol}.
 %These divisors will allow us to show $X'$ is smooth around $\spec(A)$. Indeed, by induction the $\widetilde{D}_i$ are smooth around $\spec(A)\subseteq X'$ and we may deduce smoothness by \Cref{rmk_sm_cartier_on_smooth}.
To prove (5) we will use the explicit description of a weighted blow-up of the origin in $\bA^{n}$ and the divisors $\cD_i$ and $\widetilde{D}_i$ to prove that $\cX$ is the desired weighted blow-up.

\textbf{Base step.} 
The desired statement is true if $n=1$. Indeed, by definition of a weighted projective stack when $n=1$ it follows that $\cE$ is a trivial $\bmu_n$-gerbe over $\spec(A)$. Moreover, because the ideal sheaf of the divisor restricts to $N_{\cE/\cX}$, Remark \ref{remark_cX_trivial_gen_stab} implies the only stabilizers in a neighbourhood of $\cE$ in $\cX$ are those on $\cE$. Since $\spec(A)$ is smooth, so is $\cE$, and by \Cref{rmk_sm_cartier_on_smooth} it follows that $\cX$ is smooth in a neighbourhood of $\cE$.
By applying \cite[Thm. 1]{GS17} one may obtain the stack $\cX$ in a neighbourhood of $\cE$ as a composition of root stacks and canonical covering stacks from its coarse moduli space, $X$. As non-trivial canonical covering stacks introduce new stabilizers in codimension at least 2, it follows that $\cX$ is, in a neighbourhood of $\cE$, a root stack of $X$ along a Cartier divisor $E$. The smoothness of $\cX$ in
a neighbourhood of $\cE$ implies the smoothness of $E$ since the associated cyclic cover is a local atlas for $\cX$ (see, e.g., \cite[Prop. 3.15]{javanpeykar_loughran_mathur_2022}). Moreover, since $\cE$ is a gerbe over the smooth (Weil) divisor $\spec A \subset X$, it follows that $E=\spec A$, $\spec A$ is Cartier in $X$, and therefore $X$ is smooth by Remark \ref{rmk_sm_cartier_on_smooth}. We may now assume $n>1$.

%Observe also that $\cP = \spec(A)\times \cP(a_1)\to \spec(A)$ has a section, so there is a morphism $\spec(A)\to \cX$, inducing $\spec(A)\to \cX\to X$ which factors via $E\to X$. As $\spec(A)\to X$ is injective, also $\spec(A)\to E$ is injective, hence birational\Gio{char 0} as both are of the same dimension. So from Zariski's main theorem, $\spec(A)\to E$ has to be an isomorphism. We now assume $n>1$.

\textbf{Step 1.} In this step we establish a useful vanishing result.

Consider $H$ a very ample Cartier divisor on $X$ such that $\oH^i(\cO_X(kH))=0$ for $k,i>0$. Let then $m$ a positive integer such that, if we denote by $\cH$ the pull-back of $H$ to $\cX$, we have that
$\cG:=\cO_\cX(\cH)\otimes\cO_{\cX}(m\cE)$ is the trivial line bundle when restricted to any fiber of $\cE\to \spec(A)$. 
This is possible as $\cN_{\cE\vert\cX}\cong\cO_\cP(-1)$. Moreover $m>0$ as $H$ is very ample, and up to replacing $H$ with some multiple of it we can assume $m>a_n+1$. Up to replacing $\cB$ with $\cU$ as in the statement, we can also assume that $$\cG|_\cE\cong \cO_{\cE}.$$ 
{Indeed, consider $\cE\to \spec(A)$, let $\tau:\cE\to E$ be its coarse moduli space, and let $\zeta:E\to \spec(A)$ be the morphism induced by $\cE\to \spec(A)$. As $\cG$ restricts to the trivial line bundle on any fiber of $\cE\to \spec(A)$, from \cite[Cor. 4.3]{noohipicard} there is a line bundle $\cN$ on $\spec(A)$ such that $\cG=\zeta^*\cN$. If we denote by $\rho:\spec(A)\to \cB$ the map in the statement of the theorem, as $\rho$ is finite the ring $A\otimes_BB_{\rho(\xi)}$ is semilocal. So $\cN|_{\spec(A\otimes_BB_{\rho(\xi)})}$ is trivial as semilocal rings have trivial Picard group. Then there is an open subscheme $\cU_{\spec(A)}$ containing $\rho^{-1}(\rho(\xi))$ such that $\cN$ is trivial when restricted to $\cU_{\spec(A)}$. We can take \[\cU=\cB\smallsetminus\rho(\spec(A)\smallsetminus \cU_{\spec(A)}).\] }

Consider now $\cG^{\otimes k}\otimes \cO_{\cX}(-j\cE)$, for $k> 0$. We have the following exact sequence:
$$0\to \cG^{\otimes k}\otimes \cO_{\cX}(-(j+1)\cE) \to \cG^{\otimes k}\otimes \cO_{\cX}(-j\cE)\to \cO_{\cP}(j)\to 0.$$
From \cite[Prop. 2.5]{meier2015vector}, $\oH^1(\cO_{\cP}(j))=0$ for every $j\ge -a_1 - a_2 - ... - a_n + 1$ so combining this with the sequence above we see that 
\begin{equation}\label{eq_1}
[\oH^1(\cG^{\otimes k}\otimes \cO_{\cX}(-(j+1)\cE))=0] \Longrightarrow [\oH^1(\cG^{\otimes k}\otimes \cO_{\cX}(-j\cE))=0] \text{ for }j\ge -a_1 -  ... - a_n + 1.
\end{equation}

%If we set $i=-j$, we have that %\begin{equation}\label{eq_1}
 %   \text{ as long as }i\le a_1+...+a_n - 1\text{, if }i\text{ decreases then }\dim(\oH^1(\cG^{\otimes k}\otimes \cO_{\cX}(i\cE)))\text{ grows.}
%\end{equation}
Moreover, by our choice of $H$ we have
\begin{equation}\label{eq_2}
    0 = \oH^1(\cO_X(kH)) = \oH^1(\cG^{\otimes k}\otimes \cO_{\cX}(-km\cE))\text{ for every }k\ge1.
\end{equation} Since $m> a_n + 1$,  we have $-km<-a_n-1$. Thus, combining \Cref{eq_1} and \Cref{eq_2} we have
$$\oH^1(\cG^{\otimes k}\otimes \cO_{\cX}(i\cE))=0\text{ for every }-a_n - 1 \le i \le a_1 + ... + a_n - 1\text{ and }k\ge 1.$$

\textbf{Step 2.} The line bundle $\cG$ descends to a base-point free line bundle on $X$, and if we denote by $X'=\operatorname{Proj}(\bigoplus_d \oH^0(\cG^{\otimes d}))$, the map $E\to X \to X'$ factors via $E\to \spec(A)\to X'$, where the second map is a closed embedding;
and $X\to X'$ is an isomorphism away from $E$. In particular, we can assume that $\cG$ is the pull back of a very ample divisor on $X'$.

To check that $\cG$ descends to a line bundle on $X$ is equivalent to checking that the action of the stabilizers of the points of $\cX$ are trivial along the fibers of $\cG$ (see, e.g., \cite[Thm. 10.3]{Alp}). This holds away from $\cE$ as on that locus $\cG$ is the pull-back of a line bundle on $X$, and holds on $\cE$ since $\cG|_\cE\cong \cO_{\cE}$.

To check that $\cG$ is base point free, consider the exact sequence 
$$0\to \cG\otimes \cO_{\cX}(-\cE)\to \cG \to \cG|_{\cE}\to 0.$$
From the previous vanishing, the map $\oH^0(\cX,\cG)\to \oH^0(\cX,\cG|_{\cE})$ is surjective, so there is a section not vanishing along $\cE$ since $\cG|_\cE\cong \cO_\cE$. Moreover, $\cG\otimes \cO_{\cX}(-m\cE)$ is the pull back of a very ample line bundle, so for every $p\in \cX\smallsetminus \cE$ there is a section of $\cG\otimes \cO_{\cX}(-m\cE)$, and hence of $\cG$,
not vanishing at $p$. So $\cG$ descends to a base point free line bundle on $X$, the corresponding morphism $X\to \operatorname{Proj}(\bigoplus_d\oH^0(\cG^{\otimes d}))$ is an isomorphism away from $\cE$ (since $\cG$ is very ample away from $\cE$), and it restricts to $\cO_{\cE}$ on $\cE$, so it factors through $\cE\to \spec(A)$ and induces a closed immersion $\spec(A)\to \bP(\oH^0(\cG))$ by Lemma \ref{lem:surjobvi} since $H^0(\cG) \to H^0(\cG|_{\cE})=A$ is surjective. 
Observe also that, as $\cX$ is smooth in a neighbourhood of $\cE$ then $X$ is normal in a neighbourhood of $E$, and similarly $X'$  normal in a neighbourhood of $\spec(A)$. \footnote{We are using the fact that if $X$ is a normal projective variety and $\cL$ is a base-point free line bundle on it, then also $\operatorname{Proj}\left(\bigoplus_m \oH^0(\cL^{\otimes m})\right)$ is normal. Indeed, if we denote by $s:X\to X'$ the Stein factorization of $X\to \bP(\oH^0(\cL))$, then (1) $X'$ is normal since $X$ is normal; (2) the map $\pi:X'\to \bP(\oH^0(\cL))$ is finite hence projective, so $X'=\operatorname{Proj}\left(\bigoplus_m \oH^0(\pi^*\cO(m))\right)$; (3) $s^*\pi^*\cO(1)\cong \cL$ by construction and (4) by projection formula, since $s_*\cO_X=\cO_{X'}$, we have $\oH^0(\pi^*\cO(m))=\oH^0(\cL^{\otimes m})$. }

\textbf{Step 3.} Construction of $\cD_i$ and inductive step. 

From the previous steps, we have a surjection
$$ \oH^0(\cG\otimes \cO_{\cX}(-a_i\cE))\to \oH^0(\cO_{\cE}(a_i)) \text{ for every }i.$$

Therefore there is a global section $\cD_i$ of $\oH^0(\cG\otimes \cO_{\cX}(-a_i\cE))$ that restricts to $x_i=0$ on $\cE\cong \cP(a_1,...,a_n)$. 
In particular:
\begin{enumerate}
    \item $\cD_i\cap \cE$ is a weighted projective space over $\spec(A)$ (namely, $\cP(a_1,...,\widehat{a_i},...,a_n)\times \spec(A)$), and
    \item The normal bundle of $\cD_i\cap \cE$ in $\cD_i$ is $\cO_{\cP(a_1,...,\widehat{a_i},...,a_n)}(-1)$.
\end{enumerate} 
By induction, there is a scheme $\widetilde{D}$ with a closed embedding $q:\spec(A)\to \widetilde{D}$, such that $\cD_1$ is the weighted blow-up of $q$ with weights $a_2,...,a_n$, and $\widetilde{D}$ is smooth around $q$.
Moreover, the curves contracted by $D_1\to \widetilde{D}$ are the ones contracted by $D_1\to X'$, where $D_1$ is the coarse moduli space of $\cD_1$. So the morphism $\cD_1\to \cX\to X\to X'$ factors as $\cD_1\to \widetilde{D}\to X'$ from \cite{debarre_book}*{Lem. 1.15}, with $\widetilde{D}\to X'$ finite as there are no curves in the fibers. In particular, as $\cG$ is by construction the pull-back of an ample line bundle $G$ on $X'$, also the line bundle $G|_{\widetilde{D}}$ is ample, and we have that $\widetilde{D}= \operatorname{Proj}(\bigoplus_d \oH^0(G^{\otimes d}|_{\widetilde{D}})) = \operatorname{Proj}(\bigoplus_d \oH^0(\cG^{\otimes d}|_{\cD_1}))$.

\begin{bf}Step 4.\end{bf} $\widetilde{D}\to X'$ is a closed embedding, and therefore $\spec(A)\to \widetilde{D}\to X'$ is a closed embedding.

It follows from the vanishing $\oH^1(\cG^{\otimes d }\otimes \cO_{\cX}(a_1\cE))=0$ and the following exact sequence
$$ 0\to \cO_\cX(-\cD_1)\otimes \cG^{\otimes d}\cong \cG^{\otimes (d-1)}\otimes \cO_\cX(a_1\cE)\to \cG^{\otimes d} \to \cG^{\otimes d}|_{\cD_1}\to 0$$
that there is a surjection $\oH^0(\cG^{\otimes d})\to \oH^0(\cG^{\otimes d}|_{\cD_1})$ for $d\ge2$, so the map $\widetilde{D} = \operatorname{Proj}(\bigoplus_d \oH^0(\cG^{\otimes d}|_{\cD_1}))\to \operatorname{Proj}(\bigoplus_d \oH^0(\cG^{\otimes d}))=X'$ is a closed embedding. Therefore, also the composition $\spec(A)\to \widetilde{D}\to X'$ is a closed embedding.

\textbf{Step 5.} $\widetilde{D}$ is Cartier in $X'$, so $X'$ is smooth at $\spec(A)$.

We will show that $\widetilde{D}$ agrees with the zero locus $V(s)$ of a section $s$ of a line bundle away from $\spec(A)$, and that both $\widetilde{D}$ and $V(s)$ are $S_1$ in a neighbourhood of $\spec(A)$. Consider the long exact sequence
$$0 \to \cG\otimes \cO_{\cX}(-a_1\cE) \to \cG \to \cO_{a_1\cE}\to 0.$$
Consider a section $s_\cX\in\oH^0(\cG)$ which is the image of
$\cD_1\in \oH^0(\cG\otimes \cO_{\cX}(-a_1\cE))$. This vanishes along $\cD_1$ and $a_1\cE$, and by definition of $\cG$ it is the pull-back of a section $s$ on an ample divisor on $X'$.
In particular, there is a section of an ample divisor on $X'$ whose zero locus agrees with the image of $\widetilde{D}$, and which is isomorphic to $\widetilde{D}$ away from $\spec(A)$. As $X'$ is normal in a neighbourhood of $\spec(A)$, it is also $S_2$. Then both the zero locus of a section of a line bundle and $\widetilde{D}$ (which is smooth by induction) are $S_1$ in a neighbourhood of $\spec(A)$. So they are the closure of their generic points, then they agree: $\widetilde{D}$ is Cartier.

\textbf{End of the proof.}
As in Step 5 we may construct $n$ Cartier divisors $\widetilde{D}_1,...,\widetilde{D}_n$. After shrinking around $\xi$, we can trivialize the line bundles associated to $\widetilde{D}_i$ for every $i$, so there are regular functions $s_i$ such that $\spec(A)=V(s_1,...,s_n)$.
From \cref{prop_checking_wblowup_smooth_locally}, to check that $\cX\to X'$ is a weighted blow-up, one can work locally around $\xi\in X'$. The regular functions $s_i$ give a map $X'\to \bA^n$, and up to shrinking $X'$ and $\spec(A)$ in a neighbourhood of $\xi$, we can trivialize $\cG$ so $\cD_i$ are sections of $\cO_{\cX}(-a_i\cE)$.

We can consider the weighted blow-up of $\cX'$ along $\spec(A)$ as in \cite[Sec. 2]{Inc} and \Cref{blowupoftheorigin}, by taking an open inside $[\bA^{n+1}_{x_1,...,x_n,u}/\Gm]$ with weights $a_1,...,a_n,-1$. Recall that a map to $[\bA^{n+1}/\Gm]$ is the data of a line bundle $\cL$, a section of $\cL^{a_i}$ for every $i$, and a section of $\cL^{-1}$. Then the divisors $\cD_i$ and $\cE$ give a morphism $\cX\to [\bA^{n+1}/\Gm]$. Since $\pi^*\widetilde{D}_i = \cD_i + a_i\cE$, where $\pi:\cX\to X'$, and if $x_1,...,x_n$ are the coordinates on $\bA^{n}$ and $y_1,...,y_n,u$ are those on $\bA^{n+1}$, then the weighted blow-up sends $x_i\mapsto y_iu^{a_i}$; one can check that the following diagram is fibred:
$$\xymatrix{\cX\ar[d] \ar[r] & \cB\ar[d] \\ X'\ar[r] & \bA^n}$$
where $\cB$ is the weighted blow-up of $0\in \bA^n$ with weights ${a_1,...,a_n}$. 
So the map $\psi:X'\to \bA^n$ is flat at $\psi^{-1}(0)$, and therefore $\cX\to X'$ is a weighted blow-up from \Cref{flat_bc_wblowups}.
\end{proof}
\begin{Remark}
    One can check that the weighted blow-up of \Cref{thm_affine_contraction} is in fact a regular weighted blow-up.
\end{Remark}
\begin{Remark}\label{rmk_affine_case_is_the_same_as_Artin_case}
    With the assumptions of \Cref{thm_affine_contraction} and if $\cX$ is normal, the coarse moduli space $X$ admits the contraction $X\to Z$ of \Cref{teo_contraction_cms}. If we denote by $X'$ the coarse moduli space of the contraction $\cX'$ explicitly constructed in \Cref{thm_affine_contraction}, then $Z\cong X'$ by unicity of the contraction.\end{Remark}
\begin{Cor}\label{cor_conditions_are_satisfied_for_affine_cont_to_be_a_pushout}
    \textcolor{black}{Let $\pi:\cE\to \spec(A)$ be the map of \Cref{thm_affine_contraction}. Then:
    \begin{enumerate}
        \item $\cO_{\spec(A)}\to \pi_*\cO_\cE$ is an isomorphism,
        \item $R^1\pi_*(\cO_{\cX}(-n\cE)|_\cE) = 0$ for every $n>0$, and
        \item $ \cO_{(\cX')^{\wedge}} \cong  \varprojlim\left(p^{(k)}_*\cO_{nE}^{(k)} \right)$, where the completion is with respect to the ideal sheaf of $\spec(A)$ in $\cX'$.
    \end{enumerate}
    In particular, the diagram of \Cref{thm_affine_contraction} is a push-out in algebraic stacks.}
\end{Cor}
\begin{proof}
    \textcolor{black}{Condition (1) follows since $\pi$ is a weighted projective bundle, condition (2) follows since $\cO_{\cX}(-n\cE)|_\cE$ restricts to $\cO_\cP(n)$ on the fibers of $\pi$, and from \cite{meier2015vector}*{Prop. 2.5}. Point (3) follows since, from \Cref{rmk_affine_case_is_the_same_as_Artin_case}, the $\cX'$ constructed in \Cref{thm_affine_contraction} agrees, in a neighbourhood of $\spec(A)$, with the one constructed using \Cref{teo_contraction_cms}}. Finally, from \Cref{prop_pushout_criteria} the diagram of \Cref{thm_affine_contraction} is a push-out in algebraic stacks.  
\end{proof}

\begin{Cor}\label{cor_contraction_is_wblowup}
    Let $\phi:\cX\to \cZ$ be a morphism of separated \textcolor{black}{tame} Deligne-Mumford stacks, with $\cX$ smooth and $\cZ$ normal, which induces a projective morphism $f:X\to Z$ on coarse moduli spaces. Assume that there is a Cartier divisor $\cE\subseteq \cX$ such that:
    \begin{enumerate}
        \item $\cE$ is a weighted projective bundle $\cP_{Y}\to {Y}$ over a smooth scheme ${Y}$, and $\cN_{\cE\vert\cX}\cong \cO_{\cP_{Y}}(-1)$
        \item There is a diagram as follows, with $Y\to \cZ$ a closed embedding
        $$\xymatrix{\cE\ar[d] \ar[r] & \cX \ar[d] \\ {Y}\ar[r] & \cZ}$$
        \item $\cX\to \cZ$ an isomorphism away from $\cE$.
    \end{enumerate}
    Then $\cZ$ is smooth, the map $\cX\to \cZ$ is a weighted blow-up, \textcolor{black}{and the diagram above is a pushout in algebraic stacks}. 
\end{Cor}
\begin{proof}

    First observe that $\cZ$ is an algebraic space around ${Y}$, as being an algebraic space is an open condition on a separated Deligne-Mumford stack. Moreover, since being a weighted blow-up can be checked smooth locally from \Cref{prop_checking_wblowup_smooth_locally}, up to replacing $\cZ$ with an \'etale neighbourhood of $y\in Y$ and using \cite{EGAIV}*{Prop. 18.1.1}, we can assume that $\cZ=\spec(B)$ is affine, and that $\cP_{Y}=\cP(a_1,...,a_n)\times {Y}$. From \Cref{thm_affine_contraction}, there is a morphism $\cX\to X'$ which is a weighted blow-up over $\spec(B)$, and such that $X'\to \spec(B)$ is projective. By point (3) it is also birational, and quasi-finite by (2) and (3) (therefore finite), so it is an isomorphism from Zariski's Main Theorem. \textcolor{black}{Then the square above is a pushout as the assumptions of \Cref{prop_pushout_criteria} are satisfied by \Cref{cor_conditions_are_satisfied_for_affine_cont_to_be_a_pushout}.}
\end{proof}

For us, it will be useful to drop the assumption in \Cref{cor_contraction_is_wblowup} that $Y$ has to be a scheme. To do so, we will use the following two results.

\begin{Prop}\label{prop_extending_action}
With the same notations and assumptions of \Cref{thm_affine_contraction}, assume that one has a finite and \'etale group $G$ \textcolor{black}{of order prime to $char(k)$}, acting on $\spec(A)$, $\cB$, $\cE$ and $\cX$ in a way such that:
\begin{enumerate}
    \item the maps $\cE\to \spec(A)$, $\cE\to \cX$, $\spec(A)\to \cB$ and $\cX\to \cB$ are equivariant, 
    \item  $\xi\in \spec(A)$ is fixed by $G$, and
    \item the action of $G$ on $\cB$ is trivial.
\end{enumerate}
Let $\cX'$ be the \textcolor{black}{tame} Deligne-Mumford stack constructed using \Cref{thm_affine_contraction}, possibly after shrinking $\cB$, $\spec(A)$, $\cE$ and $\cX$ as in \Cref{thm_affine_contraction}.
Then, there is a unique action of $G$ on $\cX'$, such that the maps $\spec(A)\to \cX'$, $\cX\to \cX'$ and $\cX'\to \cB$ are equivariant.
\end{Prop}

\begin{proof}\textcolor{black}{This follows from the universal property of pushouts.}
\end{proof}

\begin{Cor}\label{cor_blowdown_Y_is_specA_divided_by_G}
    With the assumptions and notations of \Cref{prop_extending_action}, possibly after shrinking $\cB$, $\spec(A)$, $\cE$ and $\cX$ as in \Cref{thm_affine_contraction}, there is a normal \textcolor{black}{tame} Deligne-Mumford stack $\cX'$, separated over $\cB$, which fits in a diagram below
    $$\xymatrix{[\cP_A(a_1,...,a_n)/G]\ar[d]^{\pi_\cE} \ar[r] & [\cX/G]\ar[d]^\pi \\ [\spec(A)/G] \ar[r]^-i & \cX'}$$
    such that:
    \begin{enumerate}
    \item $i$ is a closed embedding
        \item the morphism $\pi$ is separated, and on the level of coarse moduli spaces it is projective,
        \item $\pi$ is an isomorphism away from $[\cP_A(a_1,...,a_n)/G]$, and its restriction to $[\cP_A(a_1,...,a_n)/G]$ agrees with $\pi_\cE$.
    \end{enumerate}
    Moreover,\begin{itemize}
        \item $\pi$ is a weighted blow-up, 
        \item $\cX'$ is isomorphic to $[\cZ/G]$ for a smooth \textcolor{black}{tame} Deligne-Mumford stack $\cZ$ (so $\cX'$ is smooth),
        
        \item  when $n>1$, the stack $\cX'$ is uniquely determined from $[\cX/G]$, $[\cP_A(a_1,...,a_n)/G]$ and $[\spec(A)/G]$ and the maps between them.
    \end{itemize} 
\end{Cor}
\begin{Remark}
    One can check that the weighted blow-up of \Cref{cor_blowdown_Y_is_specA_divided_by_G} is a \textit{regular} weighted blow-up.
\end{Remark}
\begin{proof}
    \textit{Existence.} If we consider the diagram above on the level of $G$-torsors, we obtain the diagram of \Cref{prop_extending_action}. From \Cref{thm_affine_contraction}, we can construct a Deligne-Mumford stack $\cZ$ which fits in the corresponding diagram. The action of $G$ extends from \Cref{prop_extending_action}, so we have a diagram as the one above by taking the quotient. The resulting stacks are separated as they are the quotient of separated Deligne-Mumford stacks by a finite group.
    %\Gio{details}\Mich{I think it should work like this: let $\cX \rightarrow \cY$ be a $G$-torsor over some base $S$ such that $\cX$ is separated over $S$. It is enough to prove that the morphism $\cY \rightarrow B_S G$ induced by the $G$-torsor is separated, as $B_S G$ is separated over $S$ because $G$ is finite. Being separated is a fpqc-local condition on the target, therefore it is enough to check that $\cX \simeq \cY \times_{B_S G} S \rightarrow S$ is separated, which is true by hypothesis. In fact, probably the same is true if we assume $\cX \rightarrow \cY$ to be proper and surjective}

    \textit{Moreover part.} Let $\cX'$ be such a Deligne-Mumford stack. We first check that it is smooth. To do so, using \Cref{prop_checking_wblowup_smooth_locally}, we replace $\cX'$ with an \'etale atlas, so we can assume that $\cX'$ is a scheme. Then from \Cref{cor_contraction_is_wblowup}, $\cX'$ is smooth and a weighted blow-up. By construction, it is a quotient stack. We check now that it is unique.
    
From \Cref{thm_DM_stack_is_determined_by_cms_and_codim1}, it suffices to check that two such $\cX'$ have isomorphic coarse moduli space, and that they have isomorphic stabilizers for the points of codimension one. It is immediate to show that they have isomorphic stabilizers for the points of codimension one as all the codimension one points of $\cX'$ are a subset of the codimension one points of $\cX$. Similarly, the coarse space of $\cX'$ is uniquely determined from \Cref{lemma_isom_coarse_moduli_space}.
\end{proof}

\section{Proof of \Cref{Thm_intro}}\label{section_pf_thm}

We state our main theorem. 

\begin{Teo}\label{theo:main} 
    Let $\cY$ be a smooth and separated \textcolor{black}{tame} Deligne-Mumford stack  and let $\pi:\cE \rightarrow \cY$ be a weighted projective bundle with positive dimensional fibers. Suppose we are given a (regular) closed embedding $\cE \hookrightarrow \cX$ of codimension $1$ (i.e. $\cE$ is a Cartier divisor in $\cX$) into a smooth and separated \textcolor{black}{tame} Deligne-Mumford stack $\cX$. Furthermore, assume that the normal bundle $\cN_{\cE\vert\cX}$ is isomorphic to $\cO_{\cE}(-1)\otimes \pi^*\cL$ for a line bundle $\cL$ on $\cY$. 
    Then there exists a diagram, which is a pushout in algebraic stacks,
    \begin{equation} \label{eqn:general}
    \begin{tikzcd}
    \cE \arrow[d] \arrow[r, hook] & \cX \arrow[d, dashed] \\
    \cY \arrow[r, dashed, hook]           & \cZ          
    \end{tikzcd}
    \end{equation}
    such that $\cY\hookrightarrow \cZ$ is a closed embedding and $\cZ$ is a smooth separated tame Deligne-Mumford stack. Moreover, $\cX \rightarrow \cZ$ is a weighted blowup of $\cY$ in $\cZ$.
    \end{Teo}
    \begin{Remark}
    One can check that the weighted blow-up of \Cref{theo:main} is a \textit{regular} weighted blow-up.
\end{Remark}
    
We outline the proof. Observe that \Cref{teo_contraction_cms} implies we have a normal separated scheme $Z$ and a contraction $X\to Z$, where $X$ is the coarse moduli space of $\cX$. Thus, we have a diagram:\begin{center}\begin{equation}
\label{eqn:cmsextend}
    \begin{tikzcd}
    \cX \smallsetminus \cE \arrow[d] &  \\
    Z \smallsetminus Y \arrow[r, hook, "\mathrm{open}"]   & Z          \end{tikzcd}
    \end{equation}
    \end{center}
where $Y$ is the coarse moduli space of $\cY$ and the vertical map is a coarse moduli space. We will find a smooth, separated, and tame Deligne-Mumford stack $\cZ$ fitting into the diagram above having $Z$ as its coarse space by applying Theorem \ref{thm_DM_stack_is_determined_by_cms_and_codim1} to the diagram (\ref{eqn:cmsextend}). We will then leverage this to construct the stacky contraction $\cX\to\cZ$. The key difficulty in applying Theorem \ref{thm_DM_stack_is_determined_by_cms_and_codim1} is showing (\ref{eqn:cmsextend}) satisfies the local conditions of Remark \ref{rem:etloccheck} and we do this by using \Cref{cor_blowdown_Y_is_specA_divided_by_G}.

%especially the uniqueness part. In particular, after replacing $Z$ with an appropriate \'etale atlas $U_Z$, we will construct a separated, \textit{smooth} Deligne-Mumford stack $\cX'$ over $U_Z$, using \Cref{cor_blowdown_Y_is_specA_divided_by_G}.
%We will show that the coarse space of $\cX'$ is $U_Z$ using Zariski's main theorem and so $U_Z$ will have finite quotient singularities as it is the coarse moduli space of a smooth Deligne-Mumford stack. 

To prove that the resulting $\cZ$ fits into a diagram as in \Cref{theo:main}, we apply \Cref{cor_blowdown_Y_is_specA_divided_by_G} again. To prove that $\cZ$ is smooth and $\cX\to \cZ$ is a weighted blow-up, we can work \'etale locally over $Z$ by \Cref{prop_checking_wblowup_smooth_locally}, so we replace $Z$ with an \'etale cover. By the uniqueness part of \Cref{thm_DM_stack_is_determined_by_cms_and_codim1}, $\cX'=\cZ\times_Z U_Z$, so $\cZ\times_Z U_Z$ will be smooth, so $\cZ$ will be smooth as well, and the map will be a weighted blow-up.
Finally, the uniqueness part follows from the uniqueness part of \Cref{teo_contraction_cms} together with \Cref{thm_DM_stack_is_determined_by_cms_and_codim1}.

The following proposition allows us to work locally, and it will make \Cref{cor_blowdown_Y_is_specA_divided_by_G} easier to use.

\begin{Prop}\label{prop_reduction_to_quotient}
    With the notations of \Cref{theo:main}, let $y\in Y$. There is an \'etale neighbourhood $y \in U_Z \to Z$ such that:
    \begin{enumerate}
        \item $U_Z\times_Z \cY\cong [\spec(A)/G]$, for a finite group $G$ \textcolor{black}{of order prime to $char(k)$}
        \item $U_Z\times_Z\cX=[\cW/G]$ where $\cW$ is a smooth \textcolor{black}{tame} Deligne-Mumford stack, which is generically a scheme, and
        \item $U_Z\times_Z\cE\cong [\cP_A(a_1,...,a_n)/G]$, and $[\cP_A(a_1,...,a_n)/G]\times_{[\spec(A)/G]}\spec(A)\cong \cP_A(a_1,...,a_n)$. 
    \end{enumerate}
We summarize the situation in the following diagram:

$$
\xymatrix{\cP_A(a_1,...,a_n)\ar[dr]\ar[d] \ar[r] & \cW\ar[dr] & & &\\
           \spec(A)\ar[dr] & [\cP_A(a_1,...,a_n)/G] \ar[d] \ar[r] & [\cW/G]=\cX\times_ZU_Z\ar[r] \ar[d]& \cX\ar[d]&\\
           & [\spec(A)/G]\ar[r] \ar[d] & U_Z\ar[r]^{\text{\'etale}} \ar[d]_{\text{\'etale}} & Z&\\
           & \cY\ar[r] & Z&&}
$$
\end{Prop}
The proof will be in several steps, we begin with the following two lemmas.
\begin{Lemma}\label{lemma_extend_etale_covers}
    In the hypothesis of \Cref{theo:main}, there exists an \'etale covering $\{Z_i\}$ of $Y$ in $Z$ such that $\cY_i:=\cY \times_Z Z_i$ is of the form $[\spec (A_i)/G_i]$ for some finite group $G_i$ \textcolor{black}{of order not divisible by $char(k)$} acting on the affine scheme $\spec (A_i)$ and $\cE \times_{\cY} \cY_i$ is a trivial weighted projective bundle.
\end{Lemma}
\begin{proof}
    We consider the morphisms 
    $$ \cY \rightarrow Y \hookrightarrow Z$$
    where $\cY \rightarrow Y$ is the coarse moduli space map. First of all, for every point in $\cY$ there exists an \'etale neighbourhood $[\spec (A)/G]\rightarrow \cY$ of the point where $G$ is the automorphism group of the point, therefore finite. Furthermore, the induced morphism $\spec (A^G) \rightarrow Y$ is also \'etale. See \cite{AOV_tame}*{Thm. 3.2} or \cite{AlperHallRydh}*{Thm. 1.1}. Because weighted projective bundles are classified by special groups, up to shrinking (Zariski locally) we can assume $\cE \times_{\cY} [\spec (A) /G]$ is a trivial weighted projective bundle. 

    It remains to lift the \'etale morphism $\spec(A^G)\to Y$ to an \'etale morphism to $Z$. This can be done Zariski locally on $\spec(A^G)$ using \cite{EGAIV}*{Prop. 18.1.1}. Up to shrinking again, we get the following diagram with cartesian squares:
    $$
   \begin{tikzcd}
& \spec (A') \arrow[d, "{\rm open}", hook] \arrow[r, hook] & \spec (B) \arrow[dd, "{\rm et}"] \\
{[\spec (A)/G]} \arrow[d, "{\rm et}"] \arrow[r, "{\rm cms}"] & \spec(A^G) \arrow[d, "{\rm et}"]                       &                                \\
\cY  \arrow[r, "{\rm cms}"]                                & Y \arrow[r, hook]                                      & Z.                             
\end{tikzcd}
$$
   Because the coarse moduli space map is cohomologically affine, we get that $\spec (A') \times_{\spec(A^G)} [\spec (A)/G]$ is isomorphic to $[\spec(A_0)/G]$ where $A_0$ is a  $G$-invariant affine open of $\spec (A)$ and $A'\simeq A_0^G$.
\end{proof}

\begin{Lemma}\label{lemma_extending_torsor_nonred_structure}
    Let $\cC$ be a non-reduced \textcolor{black}{tame} Deligne-Mumford stack, and let $i:\cC^{red}\to \cC$ be its reduced structure. Then for every finite \'etale group $G$, the restriction map $\oH^1(\cC,G)\to \oH^1(\cC^{red},G)$ is an isomorphism, where the cohomology is the \'etale cohomology.
\end{Lemma}
\begin{proof}

If we denote with $G_{\cC^{red}}$ the constant sheaf on the \'etale topology of $\cC^{red}$ and with $G_{\cC}$ the constant sheaf on the \'etale topology of $\cC$, then $i_*G_{\cC^{red}}=G_{\cC}$ and $R^ni_*G_{\cC^{red}}=0$ for $n>0$.

Indeed, it suffices to check this by replacing $\cC$ with the spectrum a strict Henselian ring $\spec(A)$, and to observe that since $G$ topologically is a finite disjoint union of closed points, there is an equivalence
    $$\operatorname{Hom}(\spec(A),G) \to \operatorname{Hom}(\spec(A^{red}),G).$$
Similarly, as the quotient of a local strict Henselian ring is still local and strict Henselian, $R^ni_*G_{\cC^{red}}=0$ for $n>0$. Then the isomorphism follows from the Leray spectral sequence.
\end{proof}

%\begin{Oss}
    %\textcolor{black}{The previous lemma is essentially a consequence of the fact that $\cB G$ is \'etale over the base when $G$ is an \'etale group. Consider $\cX \rightarrow \cY$ an \'etale morphism of algebraic stacks. We know that the morphism of $1$-groupoids
    %$$\Hom(\spec A, \cX) \rightarrow \Hom(\spec A, \cY) \times_{\Hom(\spec (A/I), \cY)} \Hom(\spec(A/I), \cY)$$ is an equivalence for every infinitesimal thickening $\spec(A/I) \hookrightarrow \spec A$. Alternatively, we say that $\cX \rightarrow \cY$ is formally \'etale. See for example \cite[Corollary B.9]{RydhEt}. It follows from a standard descent argument that the equivalence above holds for any infinitesimal thickenings $\cZ \hookrightarrow \cZ'$ of algebraic stacks. }  
%\end{Oss}
\begin{proof}[Proof of \Cref{prop_reduction_to_quotient}]
Consider an open of the cover of \Cref{lemma_extend_etale_covers}. Up to replacing $Z$ with this cover, we can assume that 
\begin{enumerate}
\item $Z=\spec(B)$ is affine,
    \item $\cY=[\spec(A)/G]$,
    \item\label{pt_2} $\cE\cong [\cP_A(a_1,...,a_n)/G]$ with $[\cP_A(a_1,...,a_n)/G]\times_{[\spec(A)/G]}\spec(A)\cong \cP_A(a_1,...,a_n)$.
\end{enumerate}
Our goal will be to find an \'etale neighbourhood $U_Z$ of $\spec(A^G)$ in $Z$ such that also $\cX\times_ZU_Z$ admits a $G$-torsor.  In other terms, we need to extend the $G$-torsor over $\cE$ of (\ref{pt_2}), to a $G$-torsor over the pull-back of $\cX$ to an \'etale neighbourhood of $\spec(A^G)$ in $Z$. 
To do so, we proceed in two steps. First, we extend the $G$-torsor on the "completion" of $\cE\subseteq \cX$. Then, using \cite[Thm. 3.4]{AlperHallRydh}, we algebraize it to produce the desired $U_Z$.

 Consider the ideal $J$ defining the closed embedding $\spec(A^G) \hookrightarrow \spec(B)$ and let $Z^{[n]}$ the $n$-th infinitesimal thickenings. Observe that $Z^{[n]}$ is an affine scheme and let $B_n$ be its ring of global sections. Then we denote by $\widehat{B}:= \lim_n B_n$, which is a $J$-adic ring. Finally let $\widehat{\cX}:=\cX \times_Z \spec \widehat{B}$ and $\cE_n:=\cX \times_Z \spec B_n$. Observe that $\cE_0$ might not be isomorphic to $\cE$, however $\cE$ is the reduced substack of $\cE_n$ for every $n$ {as $\cE$ is reduced and the underlying topological spaces agree with the fiber product $Y\times_Z\cX$. Indeed $Y\to Z$ is the closed embedding of the locus where $X\to Z$ are not isomorphic (namely, the image of $\cE$ in $Z$)}.

\textbf{Step 1.} The pair $(\cE_0, \widehat{\cX})$ is coherently complete. 

    This follows from \cite[Thm. 1.4]{Olsson}.

\textbf{Step 2.}
    There is a morphism $\widehat{\cX} \rightarrow \cB G$ such that the composition $\cE \hookrightarrow \widehat{\cX} \rightarrow \cB G$ corresponds to the $G$-torsor $\cP_A(a_1,...,a_n)\to [\cP_A(a_1,...,a_n)/G]$.
    
Indeed, from Tannaka duality \cite[Thm. 1.1]{HR19},
 $$\Hom(\widehat{\cX},\cB G) = \Hom( {\rm Coh}(\cB G), {\rm Coh}(\widehat{\cX})).$$
 Since the pair  $(\cE_0, \widehat{\cX})$ is coherently complete, we have that 
 $$ \Hom({\rm Coh}(\cB G), {\rm Coh}(\widehat{\cX})) =  \Hom({\rm Coh}(\cB G), \lim_n {\rm Coh}(\cE_n))= \lim_n \Hom({\rm Coh}(\cB G), {\rm Coh}(\cE_n))  $$
 which using Tannaka duality again gives us 
 $$ \Hom(\widehat{\cX}, \cB G) = \lim_n \Hom(\cE_n, \cB G).$$

Therefore it is enough enough to find a compatible sequence of $G$-torsor for the sequence of infinitesimal thickenings $$\dots \hookrightarrow \cE_n \hookrightarrow \cE_{n+1} \hookrightarrow \dots$$
This follows from \Cref{lemma_extending_torsor_nonred_structure}. 

\textbf{End of the argument}. Consider the henselization $B^h$ associated to the closed embedding $\spec(A^G) \hookrightarrow \spec (B)$.  We denote by $\cX^h$ the fiber product $\cX \times_{Z} \spec(B^h)$. Notice that we have a factorization $\spec(A^G)\hookrightarrow \spec (\widehat{B}) \rightarrow \spec( B^h) \rightarrow \spec (B)$. It suffices to show that one can extend the $G$-torsor from $\widehat{\cX}$ to $\cX^h$. 

Let $F:\operatorname{Sch}_{B^h}\to \operatorname{Set}$, sending $T\to \spec(B^h)$ to the underlying set of the groupoid $\Hom(\cX^h \times_{\spec(B^h)} T, \cB G)$. The previous lemma implies that we have an element $\widehat{\xi} \in F(\spec \widehat{B})$ that extends the $G$-torsor $\cE_G \rightarrow \cE$ {coming from the quotient $\cP_A(a_1,\ldots,a_n)\to [\cP_A(a_1,\ldots,a_n)/G]$}. Since $\cX$ is quasi-compact over $\spec B$ and $\cB G$ is locally of finite presentation, $F$  is limit-preserving (see \cite[Prop. 4.18]{LMB}). Then from \cite[Thm. 3.4]{AlperHallRydh}, there is an element $\xi^h \in F(\spec B^h)$ (or equivalently a $G$-torsor over $\cX^h$) which restricts to the $G$-torsor $\cE_G \rightarrow \cE$. 

Finally, because $\spec B^h$ is the limit of the \'etale neighbourhoods of $\spec (A^G)$ in $\spec B$, then there exists an \'etale neighbourhood $\spec \widetilde{B} \rightarrow \spec B$ and an object $\xi_{\widetilde{B}} \in F(\spec \widetilde{B})$ which extends $\xi^h$. Let $U_Z$ be the \'etale neighbourhood $\spec \widetilde{B}$ of $y$ in $Z$ and $\cW \rightarrow U_Z\times_Z \cX$ be the $G$-torsor we constructed. Then \Cref{remark_cX_trivial_gen_stab} implies that $\cW$ has generically trivial stabilizers.\end{proof}

\begin{Cor}\label{cor_cms_has_finite_quotient_sing}
With the assumptions of \Cref{theo:main}, the normal algebraic space $Z$ and the diagram (\ref{eqn:cmsextend}) satisfy the hypothesis of Theorem \ref{thm_DM_stack_is_determined_by_cms_and_codim1}.
\end{Cor}
\begin{proof}
    As $X$ and $Z$ are isomorphic away from $Y$ and $X$ is the coarse moduli space of the smooth Deligne-Mumford stack $\cX$, by Remark \ref{rem:etloccheck} it suffices to check that $Z$ satisfies the hypothesis of Theorem \ref{thm_DM_stack_is_determined_by_cms_and_codim1} in a \'etale neighbourhood of $Y$. From \Cref{prop_reduction_to_quotient}, every point of $y \in Y \subset Z$ has an \'etale neighbourhood $c\colon U_Z \to Z$ such that if we base change (\ref{eqn:general}) along $c$ we obtain:
    $$\xymatrix{[\cP_A(a_1,...,a_n)/G]\ar[d] \ar[r] & [\cW/G]\ar[d] \\ [\spec(A)/G]\ar[r] & U_Z.}$$
    By \Cref{cor_blowdown_Y_is_specA_divided_by_G}, this diagram admits a stacky contraction, i.e. a smooth Deligne-Mumford stack $\cX'\to U_Z$, separated over $U_Z$, which fits in a diagram as below.
    $$\xymatrix{[\cP_A(a_1,...,a_n)/G]\ar[d] \ar[r] & [\cW/G]\ar[d] &\\ [\spec(A)/G]\ar[r] & \cX'\ar[r]&U_Z.}$$
    As $\cX'$ is separated and smooth, it has a coarse moduli space $X'\to U_Z$ with finite tame quotient singularities. The map $X'\to U_Z$ is proper, birational and quasi-finite, so it is an isomorphism by Zariski's main theorem. Thus, we may base change the diagram (\ref{eqn:cmsextend}) along $U_Z \to Z$ which can then be made into a commutative square:
 $$ 
    \begin{tikzcd}
    \mathrm{[}W/G\mathrm{]} \smallsetminus [\cP_A(a_1,...,a_n)/G] \arrow[d] \arrow[r, dashed, hook] & \cX' \arrow[d, dashed] \\
    U_Z \smallsetminus \spec(A^G) \arrow[r, hook]   & U_Z                  
    \end{tikzcd}
    $$
    where the vertical maps are coarse moduli maps and the horizontal maps are open immersions. Thus we may conclude by applying Remark \ref{rem:etloccheck}. \end{proof}
\begin{proof}[Proof of \Cref{theo:main}]
From \Cref{cor_cms_has_finite_quotient_sing}, we may apply \Cref{thm_DM_stack_is_determined_by_cms_and_codim1}, to obtain a smooth Deligne-Mumford stack $\cZ\to Z$ which has $Z$ as coarse moduli space and which contains $\cX \smallsetminus \cE$ as an open substack, or more precisely there is a commutative diagram 
$$ 
\begin{tikzcd}
\cX \smallsetminus \cE \arrow[r, hook] \arrow[rd, "{\rm open}", hook] & \cX \arrow[r, "{\rm cms}"]  & X \arrow[d] \\& \cZ \arrow[r, "{\rm cms}"'] & Z.          
\end{tikzcd}$$

We need to show that there is an arrow $\cX\to \cZ$ and a closed embedding $\cY\to \cZ$ which fit in the desired diagram.

    \textit{Construction of $\cX\to \cZ$.} We plan on using \Cref{lemma_extension_on_open_and_etale_neigh_of_complement}, so we replace $Z$ with $U_Z$ as in \Cref{prop_reduction_to_quotient}. So now the assumptions of \Cref{cor_blowdown_Y_is_specA_divided_by_G} apply, and we can construct the stack $\cX'$ of  \Cref{cor_blowdown_Y_is_specA_divided_by_G} which fits in the desired diagram. {By \Cref{thm_DM_stack_is_determined_by_cms_and_codim1} and \Cref{rem:etloccheck}, both $\cX'$ and $\cZ$ are uniquely determined by their coarse moduli
    spaces and an open substack with complement of codimension at least 2. They have isomorphic coarse moduli spaces from \Cref{lemma_isom_coarse_moduli_space}, and they are isomorphic away from the image of the exceptional divisor (which is closed and by assumption has codimension at least 2). As there is a map $\cX\to \cX'$, there is a map $\cX\to \cZ$ as desired. }

    %\textcolor{red}{OLD VERSION: As both $\cX'$ and $\cZ$ are uniquely determined from its coarse moduli space and $\cX\smallsetminus \cE$, and , they are isomorphic, so there is a map $\cX\to \cZ$.}
    
    \textit{Construction of $\cY\to \cZ$.} 
    Firstly, notice that we know that such a morphism exists \'etale locally on $\cY$ by construction, because the pullback of $\cZ$ through $U_Z\rightarrow Z$ coincides with $\cX'$. We need to prove that it gives a morphism $\cY \rightarrow \cZ$.
    We denote by $- \circ \pi$ the morphism over $\cY$ $$\cY \times \cZ \simeq \Hom_{\cY}(\cY,\cZ) \longrightarrow \Hom_{\cY}(\cE,\cZ)$$
    defined by precomposing with the morphism $\pi:\cE \rightarrow \cY$, and we use the shorthand $\Hom_{\cY}(\cY,\cZ)$ for $\Hom_{\cY}(\cY,\cY\times \cZ)$, and similarly $\Hom_{\cY}(\cE,\cZ):=\Hom_{\cY}(\cE,\cY\times\cZ)$.
    The composition $\cE \hookrightarrow \cX \rightarrow \cZ$ corresponds to a morphism $\phi:\cY \rightarrow \Hom_{\cY}(\cE,\cZ)$. We need to find a morphism $\cY \rightarrow \cY \times \cZ$ which lifts $\phi$, or equivalently a section of the projection 
    $$\eta:\cY \times_{\Hom_{\cY}(\cE,\cZ)} \Hom_{\cY}(\cY,\cZ) \rightarrow \cY$$ induced by the fiber product construction. Because $\pi:\cE \rightarrow \cY$ is a weighted projective stack over $\cY$, one has $\pi_*\cO_\cE\cong \cO_\cY$, so we can apply \Cref{lem:very-weak-rig} and thus $-\circ \pi$ is a monomorphism. By base change, the same is true for $\eta$. However, finding a section of a monomorphism is equivalent to proving it is an equivalence, which can be checked \'etale locally on $\cY$.

    \textit{Uniqueness of $\cZ$.} Let $\cZ'$ be a smooth and separated Deligne-Mumford stack that fits in the usual contraction diagram. Then we have that its coarse moduli space $Z'$ fits also in the contraction diagram at the level of coarse moduli spaces, i.e. there is an induced contraction diagram
    \[ 
    \begin{tikzcd}
        E \ar[r] \ar[d] & X \ar[d] \\
        Y \ar[r] & Z'.
    \end{tikzcd} \\
    \]
    From this we deduce that $Z\simeq Z'$, thanks to the uniqueness part of \Cref{teo_contraction_cms}. But now $\cZ$ and $\cZ'$ share the same coarse space, and are isomorphic up to substacks of codimension $\geq 2$. Then from \Cref{thm_DM_stack_is_determined_by_cms_and_codim1} we deduce that $\cZ\simeq \cZ'$. \textcolor{black}{The resulting square is a pushout since the conditions of \Cref{prop_pushout_criteria} are satisfied by \Cref{cor_conditions_are_satisfied_for_affine_cont_to_be_a_pushout}.}
\end{proof}

\begin{Lemma}\label{lem:very-weak-rig}
Let $f:\cX \rightarrow \cY$ be a 
 proper and flat morphism of finite presentation of algebraic stacks over a base algebraic space $S$ and let $\cZ$ be an algebraic stack with affine diagonal locally of finite presentation over $S$. Then the morphism of algebraic stacks over $\cY$
$$- \circ f: \cY\times_S\cZ\simeq\Hom_{\cY}(\cY,\cY\times_S\cZ) \longrightarrow \Hom_{\cY}(\cX,\cY\times_S\cZ)$$
is representable by algebraic spaces and locally of finite presentation.
Moreover, suppose that the morphism $$f^{\sharp}_T:\cO_{T}\rightarrow f_{T,*}\cO_{\cX_T}$$
is an isomorphism for every $T\rightarrow \cY$ where $\cX_T:=\cX \times_{\cY} T$, then $- \circ f$ is a monomorphism.

\end{Lemma}

\begin{proof}We will use the shorthand $\Hom_{\cY}(\cY,\cZ) := \Hom_{\cY}(\cY,\cY\times_S\cZ) $ and $\Hom_{\cY}(\cE,\cZ) := \Hom_{\cY}(\cE,\cY\times_S\cZ) $.
    First of all, notice that $\Hom_{\cY}(\cX,\cZ)$ is an algebraic stack locally of finite presentation with affine diagonal, thanks to \cite[Thm. 1.2]{HR19}. To see that $- \circ f$ is faithful (therefore representable), we fix two objects $g_1,g_2:T\rightarrow \cZ$ of $\Hom_{\cY}(\cY,\cZ)$ for a scheme $T\to \cY$, and two morphisms $\alpha,\alpha':g_1 \rightarrow g_2$ laying over a morphism $T\rightarrow \cY$ such that the isomorphisms $\alpha$ and $\alpha'$ agree once pulled back to $T\times_\cY\cX$.  We now have that $\alpha=\alpha'$ as they agree fppf locally, and since $\cZ$ is an algebraic stack, $\operatorname{Isom}$ is a sheaf in the fppf topology.
    %This is equivalent to finding a dotted arrow $$
    %\begin{tikzcd}
    %& T \arrow[d, "{(g_1,g_2,(\alpha,\alpha'))}"] \arrow[ld, dashed] \\
%\cZ \arrow[r, "\Delta_{\Delta_{\cZ}}"'] & \cZ \times_{\cZ \times_S \cZ} \cZ                 \end{tikzcd}$$
%such that the diagram is commutative. If we denote by $\Delta^2_T$ the fiber product of the diagram above, this is equivalent to  saying that $\Delta^2_T \rightarrow T$ is an isomorphism, because the second diagonal is always a monomorphism. Therefore it is enough to prove it fppf-locally. Because we know that it is true up to precomposing with $f_T:\cX_T \rightarrow T$ which is an fppf morphism, it is enough to take a smooth atlas of $\cX_T$ and we are done. 

We now prove that $- \circ f$ is full assuming that $f^{\sharp}_T:\cO_{T}\rightarrow f_{T,*}\cO_{\cX_T}$ is an isomorphism for every $T\rightarrow \cY$.
Suppose given two objects $g_1,g_2:T\rightarrow \cZ$ and an isomorphism of functors $\beta:g_1 \circ f_T \rightarrow g_2 \circ f_T$ over a morphism $T\rightarrow \cY$, we want to find an isomorphism $\alpha:g_1 \rightarrow g_2$ such that such that $\alpha$ pulls back to $\beta$. This is equivalent to finding a dotted arrow 
$$
\begin{tikzcd}
\cX_T \arrow[r, "f_T"] \arrow[d, "g_1 \circ f_T"'] & T \arrow[d, "{(g_1,g_2)}"] \arrow[ld, dashed] \\
\cZ \arrow[r, "\Delta_{\cZ}"']                     & \cZ \times_S \cZ                       \end{tikzcd}
$$
such that the diagram commute. Notice that we are using a specific natural transformation $$(\Id_{g_1\circ f_T},\beta):\Delta_{\cZ} \circ (g_1\circ f_T) \rightarrow (g_1\circ f_T,g_2 \circ f_T)$$ to make the diagram commute. If we denote by $\Delta_{T}$ the fiber product of the diagram above, it is enough to prove that the morphism $d_T:\Delta_{T}\rightarrow T$ admits a section. We leave to the reader to check that the compatibilities with the natural transformations follow from faithfulness. Therefore we have a sequence of morphisms 
$$ \cX_T \rightarrow \Delta_T \rightarrow T$$
and, because $\cZ$ has affine diagonal, we know that $\Delta_T\rightarrow T$ is affine over $T$ , thus given by an $\cO_T$-algebra $\cD$. Therefore we need to construct a morphism of $\cO_T$-algebra $\cD \rightarrow \cO_T$ which is a section of the structural morphism $\cO_T \rightarrow \cD$. This follows easily by considering the sequence of structural sheaves
$$ \cO_T \rightarrow d_{T,*}\cO_{\Delta_T}\simeq\cD \rightarrow f_{T,*}\cO_{\cX_T}\simeq\cO_T,$$
where the last isomorphism is exactly $(f_T^{\sharp})^{-1}$.
\end{proof}
\begin{Lemma}\label{lemma_extension_on_open_and_etale_neigh_of_complement}
    Assume that $\cX$ and $\cY$ are two separated \textcolor{black}{tame} Deligne-Mumford stacks, and assume that there is a dense open substack $\cU\subseteq \cX$ with a morphism $f_\cU:\cU\to \cY$. Assume also that $\cX$ is normal, and that each point of $\cX\smallsetminus \cU$ has an \'etale neighbourhood where $f_\cU$ extends. Then $f_\cU$ extends to a morphism $f:\cX\to \cY$. 
\end{Lemma}
\begin{proof}It suffices to check that if the aforementioned extension exists, then it is unique, as if that's the case we can use descent. The extension is unique from \cite[Lem. 7.2]{DH18}.
\end{proof}
\section{An application to moduli of stable pointed curves of genus one}\label{section_application_to_the_case_of_curves}
For basic facts on moduli of curves that are used throughout this section, the reader can consult \cite[Chapter XII]{ACGII}.

For $n\geq 2$, let $\overline{\mathscr{M}}_{1,n}$ denote the moduli stack of stable $n$-marked curves of genus one. Let $\Delta_{0,n}$ be the irreducible, smooth boundary divisor consisting of curves which are obtained by gluing a $1$-marked genus one curve to the $n+1^{\rm th}$ marking of a stable $n+1$-marked genus zero curve. In other terms, the divisor $\Delta_{0,n}$ is the image of the gluing morphism
\[ \xi_{0,n}: \overline{\mathscr{M}}_{1,1} \times \overline{\mathscr{M}}_{0,n+1} \longrightarrow \overline{\mathscr{M}}_{1,n}. \]
Recall that the Hodge line bundle $\bE\to \overline{\mathscr{M}}_{1,1}$ can be defined as follows: if $\pi:\mathscr{C}\to \overline{\mathscr{M}}_{1,1}$ is the universal curve and $\sigma:\overline{\mathscr{M}}_{1,1} \to \mathscr{C}$ is the universal section, then $\bE:=\sigma^*\omega_{\pi}$, where $\omega_{\pi}$ denotes the relative dualizing sheaf. The following fact can be found in \cite[Remark after Proposition 21]{EG}.
\begin{Prop}\label{prop:iso m11}
    Assume that the characteristic of the base field is $\neq 2,3$. Then there is an isomorphism $\varphi:\overline{\mathscr{M}}_{1,1}\simeq \cP(4,6)$ such that the Hodge line bundle $\bE\to\overline{\mathscr{M}}_{1,1}$ is identified with $\cO_{\cP(4,6)}(1)$.
\end{Prop}
In particular, we deduce that 
\[\pr_2:\Delta_{0,n} \simeq \cP(4,6)\times\overline{\mathscr{M}}_{0,n+1} \to \overline{\mathscr{M}}_{0,n+1}\]
is a weighted projective bundle. The normal bundle of $\Delta_{0,n}$ can be computed as follows: the isomorphism $\Delta_{0,n}\simeq \overline{\mathscr{M}}_{1,1} \times \overline{\mathscr{M}}_{0,n+1}$ determines two families of curves $\pi_1:\mathscr{C}\to \Delta_{0,n}$ and $\pi_2:\mathscr{D}\to \Delta_{0,n}$ of genus respectively $1$ and $0$, together with sections $\sigma:\Delta_{0,n}\to \mathscr{C}$ and $\sigma_{i}:\Delta_{0,n}\to\mathscr{D}$ for $i=1,\ldots,n+1$. Then we have
\[ \cN_{\Delta_{0,n}\vert\overline{\mathscr{M}}_{1,n}} = \sigma^*\omega_{\pi_1}^{\vee}\otimes \sigma_{n+1}^*\omega_{\pi_2}^{\vee} = \bE^{\vee} \otimes \sigma_{n+1}^*\omega_{\pi_2}^{\vee} . \]
In particular, applying \Cref{prop:iso m11} we get that the normal bundle of $\Delta_{0,n}$ is isomorphic to $\cO_{\cP(4,6)}(-1) \otimes \pr_2^*\bL_{n+1}$, where $\bL_{n+1}$ is the line bundle on $\overline{\mathscr{M}}_{0,n+1}$ given by the $n+1^{\rm th}$-marking. Therefore, we can apply \Cref{theo:main} to deduce the following.
\begin{Prop}
    Assume that the base field has characteristic zero and that $n\geq 2$. Then there exists a smooth Deligne-Mumford stack $\overline{\mathscr{M}}_{1,n}^*$ together with a closed embedding $\overline{\mathscr{M}}_{0,n+1}\hookrightarrow \overline{\mathscr{M}}_{1,n}^*$ and a contraction
    \[
    \begin{tikzcd}
        \Delta_{0,n} \ar[r] \ar[d] & \overline{\mathscr{M}}_{1,n} \ar[d, "f"] \\
        \overline{\mathscr{M}}_{0,n+1} \ar[r, "\iota"] & \overline{\mathscr{M}}_{1,n}^*
    \end{tikzcd}
    \]
    which realizes $\overline{\mathscr{M}}_{1,n}$ as a weighted blow up of $\overline{\mathscr{M}}_{1,n}^*$ along the center $\overline{\mathscr{M}}_{0,n+1}$.
\end{Prop}

Let $\overline{\mathscr{M}}_{1,n}^{\rm{ps}}$ be the moduli stack of pseudo-stable $n$-pointed curves of genus one, introduced in \cite[Sec. 5]{Sch}. This is a smooth Deligne-Mumford stack that fits into the diagram
\[
\begin{tikzcd}
    \Delta_{0,n} \ar[r] \ar[d] & \overline{\mathscr{M}}_{1,n} \ar[d, "f"] \\
    \overline{\mathscr{M}}^{\rm{ps, cusp}}_{1,n} \ar[r, "\iota"] & \overline{\mathscr{M}}_{1,n}^{\rm{ps}}
\end{tikzcd}
\]
where $\overline{\mathscr{M}}^{\rm{ps, cusp}}_{1,n}$ denotes the locus of cuspidal curves (see \cite[Prop. 3.11]{CTV}). This locus is isomorphic to $\overline{\mathscr{M}}_{0,n+1}$ via the pinching morphism \cite[Sec. 2]{DLPV}
\[ \overline{\mathscr{M}}_{0,n+1} \longrightarrow \overline{\mathscr{M}}^{\rm{ps, cusp}}_{1,n} \]
that creates a cusp at the $n+1^{\rm{th}}$ marking of the genus zero curve. Therefore, we have that the moduli stack $\overline{\mathscr{M}}_{1,n}^{\rm{ps}}$ is also a contraction of $\overline{\mathscr{M}}_{1,n}$ whose exceptional divisor is $\Delta_{0,n}$, and that the induced morphism $\Delta_{0,n}$ factors through $\overline{\mathscr{M}}_{0,n+1}$. Then from the uniqueness part of \Cref{theo:main} we deduce the following.
\begin{theorem}\label{thm:elliptic}
    The moduli stack $\overline{\mathscr{M}}_{1,n}$ is a weighted blow-up with weights $4$ and $6$ of the moduli stack $\overline{\mathscr{M}}_{1,n}^{\rm{ps}}$ of pseudo-stable curves.
\end{theorem}
\begin{Remark}
    \Cref{thm:elliptic} can be generalized to the birational morphisms appearing in Smyth's alternative compactifications of the moduli stack $\overline{\mathscr{M}}_{1,n}$ (see \cite{Smy11}). In particular, at least for $n=3,4,5, 6$, our criterion can be used for proving that the morphisms relating the different Smyth's compactifications can all be described as weighted blow-ups. This in turn gives a way to compute the integral Chow rings of these stacks, leveraging the formulas contained in \cite{AO}. This application can be found in \cite{BDL}.
\end{Remark}
\appendix
\section{The $\cO(m)$ case, by Stephen Obinna}\label{section_appendix}
We will work over an algebraically closed field of characteristic 0. Let $\cX$ be a smooth and separated Deligne-Mumford stack with a Cartier divisor $\cE\to \cX$ such that $\cE$ is also a fibration $g:\cE\to \cY$ of positive dimensional weighted projective stacks $\cP(a_1,\dots,a_r)$  over a smooth and separated Deligne-Mumford stack $\cY$.
Let $\cN=\mathcal{O}_{\cX}(\cE)|_{\cE}$ 
and assume that on every fiber we have $\cN_x\cong \mathcal{O}_{\cE_x}(-m)$.
We can then take the following rootstacks:
$\cX_m=\cX(\sqrt[m]{\cE})$, and
$\cE_m=\cE(\sqrt[m]{\cN})$.
This allows us to state the main theorem:

\begin{Teo}\label{thm_appendix}
There exists a $\bmu_m$ gerbe $\cY_m\to \cY$ as well as a smooth and separated Deligne-Mumford stack $\cZ_m$ completing the diagram below, and such that $\cX_m$ can be recovered as a weighted blow-up of $\cZ_m$ at $\cY_m$ with exceptional divisor $\cE_m$.
\end{Teo}
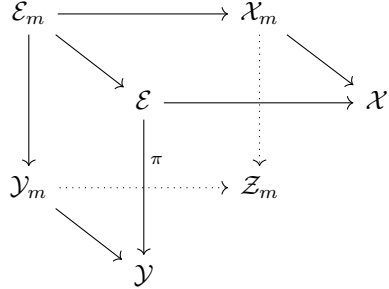
$$\begin{tikzcd}
\cE_m \arrow[rr,""] \arrow[dr] \arrow[dd,""] &&  \cX_m \arrow[dr] \arrow[dd,dotted,"",pos=0.3] \\
&\cE \arrow[rr,"",pos=0.3] \arrow[dd,"\pi",pos=0.3] && \cX  \\
\cY_m \arrow[dr] \arrow[rr,dotted,"",pos=0.3] && \cZ_m\\
& \cY 
\end{tikzcd}$$

\begin{proof}The proof is essentially in two steps.

\textbf{Step 1.} There is a $\bmu_m$-gerbe $\cY_m\to \cY$ which fits in the diagram above, such that $\cE_m\cong \cY_m\times_\cY\cE$. 

Indeed, recall that $\bmu_m$ gerbes over a Deligne-Mumford stack $\cW$ are classified by the \'etale cohomology $\oH^2(\cW,\bmu_m)$. Consider $\pi:\cE\to \cY$ the projection. Since $\pi_*\cO_\cE=\cO_\cY$, we have that $\pi_*\bmu_m = \bmu_m$, where the left hand side is the push-forward of $\bmu_m$ on the \'etale site of $\cE$, and the right hand side is $\bmu_m$ on the \'etale site of $\cY$. Note that $R^1\pi_*\bmu_m=0$. Indeed, by the proper base change theorem in \'etale cohomology, it suffices to check that $\oH^1(\cP, \bmu_m)=0$, where $\cP$ is a weighted projective stack over an algebraically closed field $k$. This follows from the cohomology of the Kummer exact sequence, since the Picard group of $\cP$ has no torsion. 

Therefore the low degree terms of the Leray spectral sequence associated to $\pi$ and $\bmu_m$ yields an exact sequence
\[
0\to \oH^1(\cY, \bmu_m)\to \oH^1(\cE, \bmu_m)\to 0\to \oH^2(\cY, \bmu_m)\to \oH^2(\cE, \bmu_m)\xrightarrow{\alpha} \oH^0(\cY, R^2\pi_*\bmu_m).
\]
The morphism $\alpha$ is the restriction of a class on $\oH^2(\cE,\bmu_m)$ along the fibers of $\pi$. Therefore, to prove the desired statement, it suffices to check that the root gerbe $\cE(\sqrt[m]{\cN})$, which corresponds to a class $c\in \oH^2(\cE,\bmu_m)$, has a section when restricted to the geometric fibers of $\pi$. But this is true, as along the geometric fibers of $\pi$ the line bundle $\cN_x$ is the $m$-th power of a line bundle.

\textbf{Step 2.} We check that $\mathcal{N}_{\cE_m|\cX_m}$ restricts to $\cO(-1)$ along each fiber of $\pi$.

Observe that the fibers of $\cE_m\to \cY_m$ are weighted projective stacks, as it is true for $\pi$ and $\cE_m=\cY_m\times_\cY\cE$.
Let $\cN_m$ be the pull-back of $\cN$ to $\cE_m$.
By how root stacks along a divisor are constructed, and since $\cN=\cO_\cX(\cE)|_\cE$, we have that \[
\cN_{\cE_m\vert\cX_m}^{\otimes m}\cong (\cN_m)|_{\cE_{m}}.\] Or in other terms, $\cN_m$ is isomorphic to the $m$-th power of the normal bundle of $\cE_m$ in $\cX_m$. Then:
\begin{enumerate}
    \item $\cN$ restricts to  $\mathcal{O}(-m)$ along each fiber of $\cE\to \cY$, so $\cN_m$ restricts to  $\mathcal{O}(-m)$ along each fiber of $\cE_m\to \cY_m$, and
    \item the Picard group of the fibers of $\cE_m\to \cY_m$ is torsion free.
\end{enumerate}
Then $\cN_{\cE_m\vert
\cX_m}$ restricts to $\mathcal{O}(-1)$ along the fibers of $\cE_m\to \cY_m$.
Now \Cref{theo:main} applies.
\end{proof}

\bibliographystyle{amsalpha}
\bibliography{bibliography}
\end{document}